\documentclass{article}

\usepackage[utf8]{inputenc}
\usepackage[T1]{fontenc}

\usepackage{amsthm}
\usepackage{amsmath}
\usepackage{amssymb}
\usepackage{mathrsfs}
\usepackage{mathtools}
\usepackage{tikz}
\usetikzlibrary{cd}

\newcommand{\C}{\mathbb{C}}
\newcommand{\R}{\mathbb{R}}
\newcommand{\N}{\mathbb{N}}
\newcommand{\Z}{\mathbb{Z}}
\newcommand{\V}{\mathbb{V}}
\newcommand{\del}{\partial}
\newcommand{\bigslant}[2]{{\raisebox{.2em}{$#1$}\left/\raisebox{-.2em}{$#2$}\right.}}

\newtheorem{defn}{Definition}[section]
\newtheorem{lem}[defn]{Lemma}
\newtheorem{thm}[defn]{Theorem}
\newtheorem{cor}[defn]{Corollary}
\newtheorem{prop}[defn]{Proposition}

\newtheorem*{rem}{Remark}

\newtheorem*{nota}{Notation}

\newtheorem{Intro}{Theorem}

\title{$L^2$ cohomology of a Variation of Hodge structure for an infinite covering of an open curve ramified at infinity}
\author{Bastien Jean}

\begin{document}

\maketitle

\begin{abstract} Let $X$ be a compact Riemann surface, $\Sigma$ a finite set of points and $M = X\setminus \Sigma$. We study the $L^2$ cohomology of a polarized complex variation of Hodge structure on a Galois covering of the Riemann surface of finite type $M$. In this article we treat the case when the covering comes from a branched covering of $X$, and where $M$ is endowed with a metric asymptotic to a Poincaré metric. We prove that after tensorisation with the algebra of affiliated operators, the $L^2$ cohomology admits a pure Hodge structure.
\end{abstract}

\section*{Introduction}

Hodge theory has been historically linked with $L^2$ cohomology. For a compact Kähler manifold it is well known that its $L^2$ cohomology is naturally isomorphic to its De Rham cohomology, and one can use the structure of Hilbert space to define the Laplace operator. Moreover every cohomology class is represented by a unique harmonic form and it is this fact that allowed the first proof of the Hodge decomposition of the cohomology of a compact Kähler manifold.

If we look at a covering space, the $L^2$ De Rham cohomology group of a Galois covering $\pi\colon\tilde X \to X$ of a complete complex manifold of dimension $n$ can be expressed as a direct sum
\[
H^k_2(\tilde X, \C) = Harm^k(\tilde X, \C) \oplus \frac{\overline{Im(d)}}{Im(d)}
\]
where $Harm^k(\tilde X, \C)$ is the space of square integrable harmonic $k$-form, and $\frac{\overline{Im(d)}}{Im(d)}$ is the torsion part. One can notice that in this setting, unless the torsion part vanishes, it is no longer true that any closed form can be represented by a square integrable harmonic form. We have an action the Deck transformation group $\Gamma$ of the covering on the space of harmonic forms of $\tilde X$ which gives them a structure of a module over the Von Neumann algebra $\mathcal N (\Gamma)$ of the groups $\Gamma$, this allows to define a Von Neumann dimension on those spaces. The $L^2$ index theorem due to Atiyah \cite{Atiyah} states that if $X$ is compact then the $L^2$ Euler characteristic of $\tilde X$, that is the alternate sum of the Von Neumann dimensions, is equal to the Euler characteristic of $X$.
\[
\chi_{2,\Gamma}(\tilde X , \C) := \sum_{k=0}^{2n} \dim_{\mathcal N(\Gamma)} Harm^k(\tilde X, \C) = \sum_{k=0}^{2n} \dim Harm^k(\tilde X, \C) =: \chi_{2}(X,\C).
\]

When $X$ is compact Kähler, the usual theory of harmonic forms gives us a Hodge structure of weight $k$ on the space of harmonic $k$-forms, however the torsion part needs not carry one. It was noted by Dingoyan \cite{Dingoyan} that if $\mathcal U(\Gamma)$ is the ring of operators affiliated to $\mathcal N(\Gamma)$ the torsion part is of $\mathcal U (\Gamma)$-torsion i.e
\[
\mathcal U (\Gamma)\otimes_{\mathcal N(\Gamma)} \frac{\overline{Im(d)}}{Im(d)} = 0.
\]
Thus one obtains a pure Hodge structure on $\mathcal U(\Gamma)\otimes H^k_2(\tilde X, \C)$. Moreover a $L^2$ direct image functor $\mathcal F \mapsto \mathcal F _{(2)}$ over the category of coherent sheaves has been defined by Campana-Demailly \cite{Campana-Demailly} and Eyssidieux \cite{Eyssidieux00}, and functoriality properties at the level of differentials operators has been studied in \cite{Eyssidieux22}.  With this setting Dingoyan proved in \cite{Dingoyan}
\[
Gr^p_F\left(\mathcal U (\Gamma)\otimes_{\mathcal N(\Gamma)} H^k_2(\tilde X, \C)\right) = \mathcal U (\Gamma)\otimes_{\mathcal N(\Gamma)} \mathbb H^{k-p}_2(X, \Omega^p(\tilde X)_{(2)})
\]
where $\Omega^\bullet(\tilde X)$ is the holomorphic De Rham complex on $\tilde X$.

If one consider a polarized complex variation of Hodge structure $(\V,F^\bullet,\bar{F}^\bullet, h)$ of weight $w$ on $X$, one has in particular a smooth hermitian bundle on $X$ with a flat connection. It makes sense to talk about its $L^2$ De Rham complex of its pull-back and its $L^2$ cohomology groups $H^k_2(\tilde X, \pi^*\V, \omega_X,h)$. One can expect some generalization of the results of Dingoyan and expect a pure Hodge structure of weight $w+k$ on the spaces $\mathcal U (\Gamma)\otimes H^k_2(\tilde X, \pi^*\V, \omega_X,h)$. This is a straight forward consequence of \cite{Eyssidieux22} and \cite{Dingoyan}. 

In a recent paper \cite{Eyssidieux22} Eyssidieux extends this theory and defines an $L^2$ cohomology of polarizable Mixed Hodge module. More precisely, in the case where $X$ is a smooth compact projective variety he defines a functor
\[
L^2dR \colon D^bMHM(X) \to D^bE_f(\Gamma)
\]
where $MHM(X)$ is the category of mixed Hodge module on $X$ and $E_{f}(\Gamma)$ is an abelian category in which the additive category of Hilbert $\mathcal N(\Gamma)$-modules is naturally embedded. If $\mathbb M$ is a mixed Hodge module, the $\mathcal U(\Gamma)$ module $\mathcal U(\Gamma)\otimes_{\mathcal N (\Gamma)} H^q(L^2dR(\mathbb M))$ have natural filtrations  induced by Saito's Hodge filtration $F^\bullet$ and the weight filtration $W_\bullet$ on $\mathbb M$, and Eyssidieux conjectured that those filtrations endow $\mathcal U(\Gamma)\otimes_{\mathcal N (\Gamma)} H^q(L^2dR(\mathbb M))$ with a mixed Hodge structure.

In this paper we begin the study of the one dimensional case and where $\mathbb M$ is a pure polarized Hodge module. We consider an embedding $j\colon M \to X$ of a Riemann surface $M$ into a compact complex curve $X$ such that $\Sigma := X \setminus M$ is a finite set of points. It is well known (see \cite[chapter 7]{Sabbah} that the study of polarizable Hodge module on $X$ with singularities at most at $\Sigma$ amounts to the study of the $L^2$ cohomology of polarized variation of Hodge structure on $M$ endowed with a Poincaré metric $\omega_{Pc}$, i.e if $p\in \Sigma$ there exists a neighbourhood $U$ of $p$ and a holomorphic quasi-isometry
\[
\phi \colon (U\cap M, \omega_{Pc}) \to (\Delta^*_r, \frac{dz\otimes d\bar z}{|z|^2(\ln|z|)^2})
\]
where for $0<r<1$, $\Delta^*_r = \{z\in \C \mid 0<|z|<r \}$.

Indeed by the results of Zucker \cite{Zucker79} one has $H^k_2(M, \V,\omega_{Pc},h_{Hodge}) \simeq H^k(X,j_*\V)$, which gives $H^k(X,j_*\V)$ a pure Hodge structure of analytical origin. It is to be noted that this result admits some generalizations in higher dimension, often under some extra-assumptions on the local system $\V$, the reader can refer to \cite{Cattani-Kaplan-Schmid, Kashiwara-Kawai} and more recently the paper of T. Mochizuki \cite{Mochizuki}. It is likely that this will lead to a proof of the main conjecture of \cite{Eyssidieux22}.

We will consider a Galois covering $\pi\colon \tilde M \to M$ satisfying the following property: if $p\in \Sigma$ and $U$ is a neighbourhood of $p$ taken as before then $\pi^{-1}(U)$ is bihomolorphic to a disjoint union of punctured disks. It is equivalent to ask that it is the restriction to $M$ of some ramified covering $\pi\colon \tilde X \to X$ and in a forthcoming work we will study the general one dimensional case. The purpose of this paper is to adapt Zucker computations to prove the one dimensional case of \cite[Conjecture 3]{Eyssidieux22}. We define a weakly constructible sheaf of $\mathcal N(\Gamma)$-module $\ell^2\pi^*\V$ on $X$ whose definition follows the one from \cite{Eyssidieux22}; this sheaf will be called the sheaf of square integrable sections of $\V$ and we will prove the following (see \ref{Pclem}).
\begin{Intro}\label{thm 1}
\[
H^\bullet(X,\ell^2\pi^*\V) = H_2^\bullet(M,\pi^*\V, \omega_{Pc}, h)
\]
\end{Intro}
We can define the functor $L^2$ direct image for coherent sheaves in our setting in the same fashion that \cite{Eyssidieux22}. This will allows us to consider the holomorphic $L^2$-De Rham complex $\Omega^\bullet(\pi^*\V)_{(2)}$ of our variation of Hodge structure. This complex is endowed by a filtration $F^\bullet$ induced by the Hodge filtration. We will prove the following via a $L^2$-Dolbeault lemma
\begin{Intro}\label{thm 2}
\[
\begin{aligned}
\mathbb{H}^k(X, \Omega^\bullet(\pi^*\V)_{(2)}) &\simeq H^k_2(\tilde M, \pi^*\V,\pi^*\omega_{Pc},\pi^*h)\\
\mathbb{H}^k(X, (Gr^P\Omega^\bullet(\pi^*\V))_{(2)}) &\simeq H^k(L^2Dolb^{P,\bullet}(\tilde M, \pi^*\V)).
\end{aligned}
\]
\end{Intro}
Following the ideas of Dingoyan we prove that the torsion part of the $L^2$ cohomology is of $\mathcal U (\Gamma)$-torsion and obtain a pure Hodge structure on the cohomology groups up to $\mathcal U (\Gamma)$-torsion.
\begin{Intro}\label{thm 3}
The $\mathcal U (\Gamma)$-module $\mathcal U (\Gamma) \otimes_{\mathcal N (\Gamma)} H^k(X,\ell^2\pi^*\V)$ admits a pure Hodge structure of weight $w+k$.
\end{Intro}

In sections $1$ and $2$ we give some survey of existing result of the $L^2$ De Rham complex and $L^2$-Dolbeault complex of polarized complex variations of Hodge structure, the study is done on arbitrary Kähler manifold on section $1$, and section $2$ is dedicated to the case studied in \cite{Zucker79} and \cite[Chapter 6]{Sabbah} of a smooth open algebraic curve with a Poincaré type metric at infinity as the computation will be useful later on. 

In sections $3$ and $4$ we describe this hypothesis of our covering and give the definition of the complex of sheaves $\Omega(\pi^*\V)_{(2)}^\bullet$, $\mathcal L ^2DR(\pi^*\V)$ and $\mathcal L ^2 Dolb^{P,\bullet}$ those sheaves will be defined on $X$, the compactification of the basis of our covering. The complex $\Omega(\pi^*\V)_{(2)}^\bullet$ will be the $L^2$ direct image image of the coherent sheaves $\Omega^\bullet(\tilde M)\otimes \pi^*\V$, and the complexes $\mathcal L ^2DR^\bullet(\pi^*\V)$ and $\mathcal L ^2 Dolb^{P,\bullet}(\pi^*\V)$ will be sheafitication of the Hilbert complex $L ^2DR^\bullet(\tilde M, \pi^*\V)$ and $L ^2 Dolb^{P,\bullet}(\tilde M, \pi^*\V)$.

In sections $5,6$ and $7$ we prove our $L^2$-Poincaré lemma and the $L^2$ Dolbeault lemma which gives the proof of the Theorems \ref{thm 1} and \ref{thm 2}.

Section $8$ will be a quick survey on the Von Neumann algebra $\mathcal N (\Gamma)$ of a group $\Gamma$ and its algebra of affiliated operators $\mathcal U (\Gamma)$, we will recall the definition of $\mathcal N (\Gamma)$-dimension on Hilbert $\mathcal N (\Gamma)$-module, we will also recall the notion of $\Gamma$-Fredholm complexes and show it is invariant by morphism of complexes of Hilbert $\mathcal N(\Gamma)$-module that are algebraic quasi-isomorphisms.

Section $9$ is dedicated to study the relation between the \v{C}ech complex $C^\bullet(\mathfrak U , \ell^2\pi^*\V)$ and the $L^2$-De Rham complex $L^2DR^\bullet(\tilde M, \pi^*\V)$, this will show that the $L^2$ De Rham complex is $\Gamma$-Fredholm and by applying the result of Dingoyan \cite{Dingoyan} we will prove \ref{thm 3}.

Section $10$ is dedicated to prove some version of the $L^2$ index theorem in this context. Since the basis of our covering is not compact and $\tilde M$ has not bounded geometry we cannot apply the existing results (see \cite{Atiyah}, \cite{Cheeger-Gromov}) which would give
\[
\chi_{2,\Gamma}(\tilde M, \pi^*\V) = \chi_2(M, \V)
\]
The case of finite covering shows us that we should look into a formula more similar to Riemann-Hurwitz. This can be done using the \v{C}ech complex and we will prove the following $L^2$ Riemann-Hurwitz.
\begin{Intro}($L^2$ Riemann-Hurwitz lemma)
We have
\[
\chi_{2,\Gamma}(\tilde M, \pi^{-1}\V) - \sum_{p\in \Sigma} \frac{\dim (j_*\V)_p}{n_p} = \chi_{2}(M, \V) - \sum_{p\in \Sigma} \dim (j_*\V)_p.
\]
Where if $p \in \Sigma$ is a puncture and $n_p $ is the order of the stabilizer of a preimage of $p$, and $(j_*\V) _p$ denotes the stalk at $p$ of the constructible sheaf $j_*\V$.
\end{Intro}

Finally in the last section we will express our result in the language of polarized Hodge module and prove the conjecture of Eyssidieux in our setting
\begin{Intro}
Take $X$ a compact Riemann surface and $\mathbb M$ a polarized Hodge module on $X$ of weight $w$, then $\mathcal U(\Gamma)\otimes H^p(L^2dR(\mathbb M))$ admits a Hodge structure of weight $w+p$ in the category of $\mathcal U (\Gamma)$-module. This Hodge structure is induced by Saito's Hodge filtration on the underlying $\mathcal D _X$-module.
\end{Intro}

\section*{Acknowledgement}

The author thanks his advisor Philippe Eyssidieux for his enlightening advices and for the time he took to read this paper.

\section{The $L^2$ De Rham complex of a variation of Hodge structure}

In this section, $M$ will denote a $n$-dimensional complex manifold and $\mathcal O_M$ its sheaf of holomorphic functions. If $\mathcal V$ is the sheaf of holomorphic sections of a holomorphic bundle, $\mathcal E(\mathcal V)$ will denote the sheaf of smooth sections of the underlying smooth vector bundle. The sheaves $\mathcal E^{p,q}$ will denote the sheaves of smooth forms of type $(p,q)$ on $M$. If $\mathcal V$ is a holomorphic vector bundle we will identify $\mathcal V \otimes_{\mathcal O _M} \mathcal E ^{p,q}$ with the sheaf of smooth $\mathcal V$-valued $(p,q)$-forms. We recall some basic notions about variations of Hodge structure and its relations with $L^2$ cohomology (for more generality the reader can check \cite{Cattani-Kaplan-Schmid}).

\begin{defn}
A complex polarized variation of Hodge structure of weight $w$ on $M$ is a family $(\V,F^\bullet,\bar F ^\bullet,S)$ where $\V$ is local system of $\C$-vector space,  $F^\bullet$ (resp. $\bar F ^\bullet$) is a decreasing filtration of the holomorphic vector bundle $\mathcal V := \mathcal O _M \otimes \V$ by holomorphic subbundles (resp. by antiholomorphic subbundles) and $S$ is a flat non degenerate hermitian pairing on $M$ such that
\begin{itemize}
\item If $H^{p,q} := F^p \cap \bar F ^q$ then $\mathcal E(\mathcal V) = \underset{p+q = w}{\bigoplus} \mathcal E (H^{p,q})$.
\item If $p\neq r$ then $S(H^{p,q},H^{r,s})= 0$, and $(-1)^pS$ is positive definite on $H^{p,q}$.
\item If $D$  denotes the flat connection $d\otimes 1$ on the smooth bundle $\mathcal E(\mathcal V) := C^\infty_X\otimes \V$ then $D^{1,0}\mathcal E(F^p) \subset F^{p-1}\otimes \mathcal E^{1,0}$ and $D^{0,1}\mathcal E (\bar F ^q) \subset \bar F ^{q-1}\otimes \mathcal E^{0,1}$ where $D^{1,0}$ (resp. $D^{0,1}$) is the $(1,0)$-component (resp. $(0,1)$-component) of $D$.
\end{itemize}
\end{defn}

\begin{rem}
The bundles $H^{p,q}$ have a structure of holomorphic vector bundles, the holomorphic structure is obtained via the isomorphism
\[
H^{p,q} \simeq Gr^p_F \mathcal V := \bigslant{F^p}{F^{p+1}}.
\]
The polarisation also allows us to recover the filtration $\bar F ^\bullet$ from the filtration $F^\bullet$ via the relation
\[
\bar F ^{w-p+1} = (F^p)^\bot.
\]
\end{rem}

The hermitian metric $h=\bigoplus (-1)^pS_{|H^{p,q}}$ is called the Hodge metric. The connection $D$ induces maps

\begin{itemize}
\item $d' \colon H^{p,q}\otimes \mathcal E^{r,s} \to H^{p,q}\otimes \mathcal E^{r+1,s}$
\item $d'' \colon H^{p,q}\otimes \mathcal E^{r,s} \to H^{p,q}\otimes \mathcal E^{r,s+1}$
\item $\nabla' \colon H^{p,q}\otimes \mathcal E^{r,s} \to H^{p-1,q+1}\otimes \mathcal E^{r+1,s}$
\item $\bar \nabla' \colon H^{p,q}\otimes \mathcal E^{r,s} \to H^{p+1,q-1}\otimes \mathcal E^{r,s+1}$
\end{itemize}

The map $\nabla'$ is called the Gauss-Manin connection. We set  
\[
\begin{aligned}
D' = d' + \bar \nabla' \qquad D''= d'' + \nabla' \\
\mathcal E (\V)^{P,Q} = \underset{p+r=P, q+s = Q}{\bigoplus} H^{p,q}\otimes \mathcal E^{r,s}
\end{aligned}
\]
One has $D=D'+D''$, $D'\mathcal E (\V)^{P,Q} \subset \mathcal E(\V)^{P+1,Q}$ and $D''\mathcal E^{P,Q}(\V) \subset \mathcal E(\V)^{P,Q+1}$. One also has a decomposition of $\mathcal E^\bullet (\V)$, which is the smooth De Rham complex of $\mathcal V$, the decomposition is given by
\[
\mathcal E ^\bullet (\V) = \bigoplus_{P,Q} \mathcal E(\V)^{P,Q}.
\]
This decomposition allows us to define the Hodge filtration on the complex $\mathcal E^\bullet(\V)$ by setting
\[
F^p \mathcal E^\bullet(\V) = \bigoplus_{P\geq p} \mathcal E(\V)^{P,Q}.
\]

Now, we suppose that we are given $\omega_M$ a hermitian metric on $M$. With this metric and the Hodge metric, we can consider for $k\in \N$ the Hilbert space $L_2^k(M,\V,\omega_M,h)$ of measurable $\V$-valued $k$-form that are square integrable on $M$. We obtain an elliptic complex $L^2DR^\bullet(M,\V,\omega_M,h)$.
\[
\begin{tikzcd}
0 \arrow[r]& L^2DR^0(M, \V,\omega_M,h) \arrow[r, "D"]& \dots \arrow[r, "D"]& L^2DR^{2n}(M, \V,\omega_M,\V) \arrow[r]& 0
\end{tikzcd}
\]
where we view $D$ as an unbounded operator whose domain of definition contains the space of smooth forms with compact support. Details about elliptic complexes can be find in  \cite[\S 3]{Bruning}, we recall that the operator $D$ admits two natural closures $D_{min}$ and $D_{max}$. The operator $D_{min}$ is the minimal closure of $D$, it is defined on the space of square integrable forms $\phi$ such that there exists a sequence $(\phi_n)_{n \in \N}$ of smooth square integrable forms with compact support that converge to $\phi$ and such that the sequence $(D\phi_n)_{n\in \N}$ is a Cauchy sequence and in this case we define $D\phi$ to be the limit of the sequence $D\phi_n$. The maximal closure $D_{max}$ is defined on the space of $\V$-valued $L^2$-form $\phi$ such that $D\phi$ is $L^2$ (here $D\phi$ is computed in the sense of distributions). 

We recall that an ideal boundary condition is a closed extension $D_{bc}$ of $D$ such that $D_{min} \subset D_{bc} \subset D _{max}$, and $Im(D_{bc}) \subset Ker(D_{bc})$, in practice we will only work with the boundary condition $D_{max}$.
\begin{defn}
Given an ideal boundary condition $D_{bc}$, we define the $L^2$ cohomology for this boundary condition to be the topological vector space
\[
H^\bullet_{2,bc} := \frac{Ker(D_{bc})}{Im(D_{bc})}.
\]

We say that the cohomology is fully reduced if all cohomology groups $H^k_{2,bc}$ are Hausdorff as a topological vector spaces, or equivalently if $D_{bc}$ has a closed range. The reduced cohomology $H^\bullet_{2,bc,red}$ is the Hilbert space
\[
H^\bullet_{2,bc,red} = \frac{Ker D_{bc}}{\overline{Im(D_{bc})}}.
\]
\end{defn}

We denote by $\mathfrak d$ the formal adjoint of $D$, it is the differential operator defined on compactly supported $\mathbb V$-valued smooth forms by the identity
\[
<\mathfrak d \phi,\psi> = <\phi, D\psi> 
\] 
when $\phi,\psi$ are compactly supported smooth forms. The Laplace operator is the differential operator
\[\square_D = (D + \mathfrak d)^2.
\]
It is formally self adjoint, meaning that if $\phi$ and $\psi$ are compactly supported smooth forms one has
\[
<\square_D \phi, \psi> = <\phi, \square_D \psi>.
\]
The operators $\mathfrak d$ also admits two closures $\mathfrak d _{min}$ and $\mathfrak d _{max}$ and we have
\[
D_{min}^* = \mathfrak d _{max} \qquad D_{max}^* = \mathfrak d _{min}
\]
where $D_{min}^*$ (resp. $D_{max}^*$) denotes the adjoint of $D_{\min}$ (resp. $D_{max}$). Given an ideal boundary condition we can consider the Laplace operator for this boundary condition by setting
\[
\square_{D_{bc}} = (D_{bc} + D^*_{bc})^2.
\]
The domain of this operator is the space of forms $\omega \in Dom(D_{bc})\cap Dom(D_{bc}^*)$ satisfying $D_{bc}\omega \in Dom(D_{bc}^*)$ and $D_{bc}^*\omega \in Dom(D_{bc})$.

It is self adjoint (in the usual sense of the theory of unbounded operator in Hilbert spaces). A measurable form $\phi$ is said to be harmonic if $\square_D \phi = 0$ (here $\square_D\phi$ is again computed in the sense of distribution), such a form is smooth by elliptic regularity. It is to be noted that even though square integrable harmonic forms are in the domain of $D_{max}$ they need not be $D_{max}$-closed in general, however this will be the case if the metric is complete (for a counter-example one can consider the function $x^2-y^2$ on the Euclidean disk). We set $Harm(M,\omega_M,\V,h)$ the space of square integrable harmonic form and for a boundary condition $D_{bc}$ we set
\[
\begin{aligned}
Harm_{bc}(M,\omega_M,\V,h) &:= Ker(D_{bc})\cap Ker(D_{bc}^*)\\ 
					      &= Ker(\square_{D_{bc}}) \subset Harm(M,\omega_M,\V,h).
\end{aligned}
\]
When the metrics are clear from context, we will only write $Harm_{bc}(M,\V)$ or $Harm_{bc}$, similarly we will only write $L^2DR^\bullet(M,\V)$. A useful result is weak Hodge Decomposition \cite[Lemma 2.1]{Bruning}, sometimes also referred to as the weak Kodaira decomposition.

\begin{lem}(Weak Hodge Decomposition)
For any boundary condition $D_{bc}$ we have the following orthogonal decomposition
\[
L^2DR^k(M,\V) =  Harm_{bc}^k \oplus \overline{Im D_{bc}} \oplus \overline{Im D_{bc}^*}
\]
where $Harm_{bc}^k = Harm_{bc}\cap L^2DR^k(M,\V)$.
\end{lem}

If the metric $\omega_M$ is complete then $D_{min}=D_{max}$ (see \cite[Theorem 7.2]{Zucker79}), thus there exists a unique ideal boundary condition, in this case we will drop the subscript $bc$ in the notation of the $L^2$ cohomology group we will also make an abuse of notation by identifying $D$ and $\mathfrak d$ with their closures $D_{max}$ and $\mathfrak d_{max}$. Under this hypothesis the Laplace operator $\square_D$ is essentially self-adjoint and we will also denote its minimal closure by $\square_D$. It is defined on the the space of forms $\phi$ such that $\mathfrak d \phi \in Dom(D)$ and $D\phi \in Dom(\mathfrak d )$. In this case for any $\phi \in Dom(\square_D)$, one has the classical equality
\[
<\square_D \phi, \phi> = \|D\phi\|^2 + \|\mathfrak d \phi\|^2.
\]
From this equality it follows that harmonic forms are closed and co-closed. Still assuming the metric is complete the weak Hodge decomposition gives us
\[
L^2DR^k(M,\V) = Harm^k(M,\V) \oplus \overline{Im(D)} \oplus \overline{Im(\mathfrak d)}
\]
where $Harm^k(M,\V)$ is the space of $L^2$-harmonic $k$-forms. Thanks to this orthogonal decomposition, one has canonical isomorphisms
\[
\begin{aligned}
H^k_2(M,\V) &\simeq Harm^k(M,\V) \oplus \frac{\overline{Im(D)}}{Im(D)} \\
H^k_{2,red}(M,\V) &\simeq Harm^k(M,\V).
\end{aligned}
\]

Since the Laplace operator $\square_D$ is positive self-adjoint we have a projection valued measure $(E^k_\lambda(\square_D))_{\lambda \geq 0}$ such that 
\[\square_D = \int_{\lambda \in \R^+}\lambda dE^k_\lambda(\square_D).\]
Moreover $D$ commutes with $\square_D$ so the $DE^k_\lambda(\square_D) \subset E^{k+1}_\lambda(\square_D)$ for $\lambda \geq 0$. In other words for $\lambda \geq 0$ the inclusion $E^\bullet_\lambda(L^2DR^\bullet(M,\V)) \to L^2DR^\bullet(M,\V)$ defines a subcomplex of the $L^2$ De Rham complex. The Hausdorff condition to have a reduced cohomology can be expressed in terms of the spectrum of the Laplacian.
\begin{lem}\label{criteria of reduced cohomology}
The following assertions are equivalent:
\begin{enumerate}
\item The $L^2$ cohomology  groups are fully reduced.
\item There exists $\varepsilon > 0$ such that the $E^\bullet_0(\square_D) = E^\bullet_\varepsilon(\square_D)$.
\item For $\lambda > 0$ the inclusion of complexes $E^\bullet_0(\square_D) \to E^\bullet_\lambda(\square_D)$ is a quasi-isomorphism.
\item $0$ is isolated in the spectrum of $\square_D$.
\end{enumerate}
\end{lem}

The Hilbert spaces $L^2DR^\bullet(M,\V) := \underset{k}{\bigoplus} L^2DR^k(M,\V)$ admits an orthogonal decomposition and a Hodge filtration similar the one of $\mathcal E ^\bullet(M,\V)$ by setting
\[
\begin{aligned}
L^2DR^\bullet(M,\V) &= \bigoplus_{P,Q} (L^2DR^\bullet(M,\V))^{P,Q} \\
F^pL^2DR^\bullet(M,\V) & = \bigoplus_{P\geq p} (L^2DR^\bullet(M,\V))^{P,Q} 
\end{aligned}
\]
If one wants to obtain a Hodge decomposition on the $L^2$ cohomology, it is natural to consider the Dolbeault complex $L^2Dolb^{P,\bullet}(M,\V)$ which is the elliptic complex 
\[
\begin{tikzcd}
0 \ar[r] &L^2DR^{0,0}(Gr^P_F) \ar[r, "D'' "] & L^2DR^{1,0}(Gr^{P-1}_F)\oplus L^2DR^{0,1}(Gr^{P}_F) \ar[r, "D'' "]& \dots
\end{tikzcd}
\]
As for the De Rham complex, one has two closed extension to consider $D''_{min}$ and $D''_{max}$.
We will denote by $\mathfrak d '$ and $\mathfrak d ''$ their formal adjoints, and consider the Laplace operators $\square_{D'} := (D'+\mathfrak d ')$, $\square_{D''} := (D'' + \mathfrak d '')$, they are formally self-adjoint. If the metric $\omega_M$ is Kähler, one has the following equality (see \cite[Theorem 2.7]{Zucker79}) over the space of compactly supported smooth forms
\[
\square_D = \square_{D'}+\square_{D''} = 2\square_{D''}.
\]

If we want to extend this equality, we first need to study when the two closures $D''_{min}$ and $D''_{max}$ coincide to avoid any problems coming from ideal boundaries. One has $D'' = d'' +\nabla'$, on smooth compactly supported forms, and since $d''$ is the $\bar \del$ operator of our holomorphic vector bundle by \cite[p.92]{Andreotti} we have $d''_{min} = d''_{max}$, so if $\nabla'$ is bounded, one has $D''_{min} = D''_{max}$ (see \cite[Theorem 7.1]{Zucker79}). The same reasoning that we did previously yields that $\square_{D''}$ is essentially self adjoint, we will also denote by $\square_{D''}$ its minimal closure. In this case, we can take the minimal closure the equality of Laplacian to obtain the following.
\begin{prop}
We have the equality of closed operator of $L^2DR^\bullet(M,\V)$
\[
\square_D = \square_{D'}+\square_{D''} = 2\square_{D''}.
\]
Moreover we also have the following orthorgonal decomposition of $\square_D$
\[
\square_D |_{L_2^\bullet(M,\V)} = \bigoplus_{P+Q = w+k} \square_{D''} |_{L_2^\bullet(M,\V)^{P,Q}}.
\]
\end{prop} 
If we have $P+Q=k+w$ we set $Harm^{P,Q}(M,\V) := Harm^k(M,\V) \cap \mathcal E^{P,Q}(\V)$ the space of harmonic forms of type $(P,Q)$, and one has
\[
Harm^k(M,\V) = \bigoplus_{P+Q = k+w} Harm^{P,Q}(M,\V).
\]
As for for the case $D$ the space $Harm^{P,Q}(M,\V)$ can be identified with the reduced $L^2$ Dolbeault cohomology. This gives the following proposition.

\begin{prop}
If $\V$ a local system underlying a polarised variation of Hodge structure of weight $w$ on a complete Kähler manifold $M$, the reduced $L^2$ cohomology groups $H^k_{2,red}(M,\V)$ admits a pure Hodge structure of weight $k+w$. The component of type $(P,Q)$ is canonically isomorphic to $Harm^{P,Q}(M,\V)$. 
\end{prop}
The case of the unreduced cohomology is harder to deal with, since it is not clear that the complex $Gr_F L^2DR^\bullet(M,\V)$ is quasi isomorphic to the Dolbeault complex. It was observed in \cite{Eyssidieux22} that the smooth subcomplex is more convenient to deal with the filtrations. We recall the definition of the smooth complex $L^2DR_\infty^\bullet(M, \V)$. First we define for $k\in \N$
\[
L^2DR_1^k(M, \V) = \left\{ \phi \in Dom(D)\cap Dom(\mathfrak{d}) \mid D\phi \in Dom(\mathfrak d ),\, \mathfrak d  \phi \in Dom(D)  \right\}.
\]
And we define by induction on $j\geq 1$,
\begin{align*}
L^2DR_{j+1}^k(M, \V) := \left\{ \phi \in L^2DR_j^k(M, \V)  \right.&\mid D_{max}\square_D^j \phi \in Dom(\mathfrak d _{max}), \\
												& \left. \mathfrak d _{max} \square_D ^j\phi \in Dom(D_{max})  \right\}.
\end{align*}
Finally we define $L^2DR^k_\infty(M,\V)$ to be
\[
L^2DR_{\infty}^k(M, \V) := \bigcap_{j\in \N^*} L^2DR_j^k(M, \V).
\]
And this defines a subcomplex $L^2DR_{\infty}^\bullet(M, \V)$ of $L^2DR^\bullet$ since the operators $D$ and $\mathfrak d$ commute with the Laplacian. Using $\square_{D''}$ one can define the smooth subcomplex $L^2Dolb^{P,\bullet}_\infty(M,\V)$ in a similar way. From \cite[Lemme 5.2.3]{Eyssidieux22} and \cite[Theorem 2.12]{Bruning} we have the following.
\begin{lem}\label{quasi-iso}
Assume $\lambda'>\lambda>0$ then the inclusions
\[
\begin{aligned}
E^\bullet_\lambda(M,\V) &\subset E_{\lambda'}^\bullet(M,\V) &\subset L^2DR_\infty^\bullet(M,\V) &\subset L^2DR^\bullet(M,\V) \\
E^\bullet_\lambda(M,\V) &\subset E_{\lambda'}^\bullet(M,\V) &\subset L^2Dolb_\infty^\bullet(M,\V) &\subset L^2Dolb^\bullet(M,\V)
\end{aligned}
\]
are quasi-isomorphisms.
\end{lem}

\begin{rem}
We have a canonical isomorphism
\[
Gr^P_F L^2DR^\bullet_\infty(M,\V) \to L^2Dolb^{P,\bullet}_\infty(M,\V).
\]
\end{rem}

\section{Holomorphic $L^2$ De Rham complex of an open curve with a metric with Poincaré singularities}

In this section we review some characterisation of square integrability of sections of a polarized variation of Hodge structure on a punctured disk in terms of meromorphic extension and monodromy filtration. Those results are due to Schmid \cite{Schmid} and Zucker \cite{Zucker79} for the case of real variations of Hodge structures, the reader can refer to \cite[Chapter 6]{Sabbah} for the case of complex variation of Hodge structure, and in this section we will use the same notations as \cite{Sabbah} as much as possible.

We consider a a smooth curve $M$ embedded in a Riemann surface $X$. We assume $\Sigma := X\setminus M$ is a finite number of points and $M$ has a Kähler metric $\omega_M = \omega_{Pc}$ asymptotic to a Poincaré metric near the punctures, i.e if $p\in \Sigma$ there exists $U$ a neighbourhood of $p$ and a bi-Lipschitz isomorphism for some $r<1$
\[
(U\cap M,\omega_{Pc}) \longrightarrow \left(\Delta_{r}^*, \frac{i}{2}\frac{dz\wedge d\bar z}{|z|^2(\ln|z|)^2}\right)
\]
where $\Delta^*_{r}$ denotes the disk of radius $r$ minus the origin $0$. Such a metric always exists on $M$ and $(M,\omega_{Pc})$ is complete and have finite volume. Moreover any smooth form on $X$ has bounded norm for the metric $\omega_{Pc}$. We consider a polarized variation of Hodge structure $(\V, F^\bullet, \bar F ^\bullet, S)$ on $M$. 

\begin{prop}\label{Deligne extension}[\cite[Theorem 4.4]{Malgrange} and \cite[Proposition 5.4]{Deligne}]
Let $\Sigma$ be a discrete set of point of a Riemann surface X, and set $M:=X\setminus \Sigma$. If $(\mathcal V, \nabla)$ be a flat holomorphic vector bundle on $M$ then for all $\beta \in \R$ there exists a unique extension $\mathcal V_*^{\beta}$ of $\mathcal V$ to $X$ such that
\begin{enumerate}
\item The connection $\nabla$ has at most a logarithmic pole at points $p\in \Sigma$ with respect to the meromorphic extension associated to $\mathcal V_* = \mathcal O_X(*\Sigma)\otimes_{\mathcal O_X} \mathcal V_*^{\beta}$.
\item The real part of the eigenvalues of the residue of the connection lies in $[\beta,\beta+1[$.
\end{enumerate}
\end{prop}

We recall the basis of the construction : first we localise around a puncture $p\in \Sigma$, we have a flat bundle on $\Delta^*$ that we need to extend to on $\Delta$ (the radius of the disk is unimportant here as we do not consider the metric for the construction). We set $\pi\colon \mathbb H \to \Delta^*$, $\tau \mapsto exp(2i\pi\tau)$ a universal covering. The pull back $\pi^*\mathcal V$ is isomorphic as a flat vector bundle to $\mathbb H \times V$ where $V$ is a $n$ dimensional vector space. We have $\pi_1(\Delta^*) \simeq \Z$ and the monodromy is given by the image $T^{-1}\in GL(V)$ of $\gamma$ by the monodromy representation. One has
\[
\mathcal V = \bigslant{(\mathbb H \times V)}{\Z}.
\]
One has the decomposition into generalized eigenspaces $V = \underset{\lambda \in \C}{\bigoplus} Ker(T - \lambda)^n$. Set $e_1,\dots, e_n$ a basis of horizontal sections of $\pi^*\mathcal V$ flagged according to this decomposition. If $e_j \in Ker(T - \lambda)^n$, there exists $\alpha$ with $\beta \leq \alpha < \beta +1$ with $\lambda = e^{2i\pi\alpha}$ and on this space $T = e^{2i\pi\alpha}e^{N_\alpha}$ where $N_\alpha$ is nilpotent, one can then set
\[
\xi _j = exp\left((2i\pi\alpha + N_\alpha) \tau\right)e_j
\]
it is a holomorphic basis of sections satisfying $\xi_j(\tau+1) = T \xi_j(\tau)$, so it actually defined on $\Delta^*$ and it generates the sections at $p$. 

\begin{lem}\cite[Theorem 6.3.2]{Sabbah}
In the case where $\mathcal V$ is underlying a complex variation of Hodge structure, the monodromy has eigenvalues in $\mathbb S ^1$. In particular the eigenvalues of the residue on $\mathcal V_*^\beta$ will be real numbers.
\end{lem}

\begin{rem}
We localize and take $X = \Delta$ and $M=\Delta^*$ with coordinate $z$. The morphism of $\mathcal O _X$-module $\mathcal V_* \to \mathcal V_*$ given by the multiplication by $z$, induces an isomorphism between $\mathcal V_*^\beta$ and $\mathcal V_{*}^{\beta+1}$.
\end{rem}
\begin{nota}
If we set $\mathcal V^{>\beta} = \bigcup_{\beta' > \beta} \mathcal V_*^{\beta'}$ there exists $\beta'>\beta$ such that $\mathcal V_*^{\beta'} = \mathcal V_*^\beta$. We will set $Gr^\beta \mathcal V_* = \bigslant{\mathcal V_*^\beta}{\mathcal V_*^{>\beta}}$.
\end{nota}

If $N$ is a nilpotent endomorphism of a vector space $V$, there exists a unique increasing filtration $W_\bullet(N)$ of $V$ satisfying
\[
\begin{aligned}
NW_k \subset W _{k-2} & \\
N^k\colon Gr^W_k \to Gr^W_{-k} & \text{ is an isomorphism}.
\end{aligned}
\]
The construction of this filtration is due to Schmid \cite[\S 4]{Schmid}. Using the same notation as before each $N_\alpha$ induces a filtration $W_\bullet(N_\alpha)$ of $Ker((T - e^{2i\pi\alpha})^n)$, we denote by $W_k$ the direct sum of those filtration. The filtration $W_\bullet$ can be characterized by its asymptotic behaviour (see \cite[Theorem 6.6]{Schmid} with a shift in the filtration for the case of integral variation of Hodge structure \cite[Chapter 6]{Sabbah} for the general case).
\begin{lem}
A horizontal multivalued section $e \in V$ is in $W_k \setminus W_{k-1}$ if and only if its Hodge norm $\|e\|_h^2 \simeq Im(\tau)^k$ uniformly on any vertical strip.
\end{lem}
The operator $N_\beta$ induces a nilpotent endomorphism on $Gr^\beta \mathcal V_*$ and we will also denote by $W_k(N_\beta)$ the induced filtration, we take $p\colon \mathcal V_*^\beta \to Gr^\beta \mathcal V_*$ the projection and set as in \cite{Sabbah} $M_k\mathcal V_*^\beta = p^{-1}(W_k(N_\beta))$.

\begin{rem}
The filtration $W_k(N_\beta) \subset Ker(T-e^{2i\pi\beta})^n \subset V$ defines locally near the punctures a flat subvector bundle $\mathcal W_k(N_\beta)$ of $\mathcal V$, we take $\overline{\mathcal W _k(N_\beta)}$ its extension to $X$ in the sense of Proposition \ref{Deligne extension}, such that the eigenvalues of the residue is equal to $\beta$. Then we have an isomorphism
\[
\begin{tikzcd}
M_k\mathcal V_*^\beta \simeq \overline{\mathcal W_k(N_\beta)} + \mathcal V_*^{>\beta}.
\end{tikzcd}
\]
\end{rem}

The asymptotic behaviour of the monodromy filtration gives us the following result.
 
\begin{thm}\cite[Theorem 6.3.5]{Sabbah}
If $\xi\in \mathcal V_*^\beta$ extends to a section of $M_k\mathcal V_*^\beta$ and projects non trivially in $Gr^M_{k-1}Gr^\beta \mathcal V_*$ if and only if
\[
\|\xi\|_h^2 \sim |z|^{2\beta} |\ln(z) |^k
\]
on any angular sector. Moreover $e$ projects non trivially on $Gr^W_k$ is and only if
\[
\|\xi\|^2 \simeq |z|^{2\beta} |\ln(z) |^k.
\]
\end{thm}

With this norm approximation we can express locally square integrable forms in the neighbourhood of a puncture in terms of the minimal extension (see \cite[Proposition 4.4, Proposition 6.9]{Zucker79}, or \cite[Lemma 6.10.16]{Sabbah}).
\begin{prop}\label{caracl2}\cite[Proposition 4.4]{Zucker79}
Let $(\Omega^\bullet(\V)_{(2)},\nabla)$ be the complex of sheaves defines on $X$, such that $\Omega^k(\V)_{(2)}(U)$ is the space of holomorphic $k$-form on $M\cap U$ that are square integrable on $M\cap K$ for all $K$ compact of $U$. Then its is equal to the following complex
\[
\begin{tikzcd}
0 \ar[r]& M_0\mathcal V ^0_* \ar[r, "\nabla"]& \Omega_X^1(log \Sigma) \otimes_{\mathcal O_X} M_{-2}\mathcal V^{-1}_* \ar[r]& 0
\end{tikzcd}
\]
\end{prop}

This proposition is a consequence the result below \cite[p.433]{Zucker79}.
\begin{lem}
If $(e_1,\dots, e_n)$ is a flat multivalued basis of $\V$ flagged according to the filtration $W_\bullet$ then induced frame $(\xi_1,\dots, \xi_n) \in \mathcal V^0_*$ is $L^2$ adapted, i.e there exists a constant $C > 1$ if $(f_1,\dots, f_n)$ are measurable then
\[
C^{-1} \sum\|f_i \xi_i\|^2 \leq \left\|\sum f_i \xi_i \right\|^2 \leq C \sum\|f_i \xi_i\|^2
\]
\end{lem}

\section{Branched covering of curves}
We recall the definition of a branched covering of complex manifold, it is taken from \cite{Namba}.
\begin{defn}
A surjective holomorphic map $\pi\colon (\tilde M,\mathcal O_{\tilde M}) \to (M,\mathcal O _M)$ between complex manifolds is called a branched covering if it satisfies the following properties.
\begin{enumerate}
\item The morphism $\pi$ has discrete fibers.
\item The ramification locus 
\[
R_\pi := \left\{ p \in \tilde M \mid \pi^*\colon \mathcal O_{\tilde M, \pi(p)} \to \mathcal O_{\tilde M, p} \text{is  not an isomorphism}  \right\}
\] is a hypersurface of $\tilde M$ and its image $B_\pi = \pi(R_\pi)$ is a hypersurface of M. 
\item The induced map $\pi \colon \tilde M \setminus R_\pi \to M\setminus B_\pi$ is a topological covering.
\item Any point $p\in M$ admits a neighbourhood $U$ such that $p$ has only one antecedent in each connected component $\tilde U$ of $\pi^{-1}(U)$ and $\pi\colon \tilde U \to U$ is surjective and proper.
\end{enumerate}
\end{defn}

In the case of complex curves, it means that every point $p\in M$ admits a coordinate neighbourhood $U \simeq \Delta$ such that any connected component $\tilde U$ of $\pi^{-1}(U)$ biholomorphic to a disk $\Delta$ and in coordinates $\pi$ acts as
\[
\begin{aligned}
\pi \colon & \tilde U &\to&\, U \\
		& z &\mapsto& z^{n_{\tilde U}} 
\end{aligned}
\]
for some integer $n_{\tilde U} \geq 1$ that depends on the connected $\tilde U$.  In the following we will consider the case of a complex curve $M$ embedded in a compact Riemann surface $X$ such that $\Sigma := X\setminus M$ is a finite set of points. We will assume that we have a Galois covering $\pi\colon \tilde M \to M$, with $Deck(\tilde M/ M) = \Gamma$ such that for $p\in \Sigma$, the meridian loop $\gamma$ around $p$ has a finite order in $\Gamma$. Such coverings appear for example when we consider the restriction of a ramified covering $\pi\colon \tilde X \to X$ with $B_\pi \subset \Sigma$. In fact under any such covering is of this form as stated by the lemma below.

\begin{lem}
Let $\Gamma$ be a discrete group acting freely and properly discontinuously on a Riemann surface $\tilde M$ and let $\pi$ be the covering $\pi\colon \tilde M \to \tilde M / \Gamma =: M$ and consider an embedding $j\colon M \to X$ of $M$ into a compact Riemann surface $X$ with $\Sigma:=X\setminus M$ consisting of a finite number of points.  Assuming that for $\gamma$ a meridian loop around a point $p\in \Sigma$, $\pi^{-1}\gamma$ is a disjoint union of simple closed curves there exists an embedding of Riemann surfaces $\tilde j \colon \tilde M \to \tilde X$, with the following properties.
\begin{enumerate}
\item The action of $\Gamma$ on $\tilde M$ extends to a properly discontinuous action on $\tilde X$.
\item  $\pi$ extends to a surjective $\Gamma$-invariant holomorphic branched covering 
\[\pi \colon \tilde X \to X.\]
\end{enumerate}
In this case one has the following commutative diagram.
\[
\begin{tikzcd}
\tilde M \arrow[d, "\pi"] \arrow[r, hook, "\tilde j"] & \tilde X \arrow[d, "\pi"] \\
M \arrow[r, hook, "j"] & X
\end{tikzcd}
\]
\end{lem}
\begin{proof}
The existence of a topological space $\tilde X$ such that we have the commutative diagram is a particular case of Fox completion \cite{Fox}. By construction of the Fox completion $\pi\colon \tilde X \to X$ is surjective, has discrete fibers and any point $p\in X$ admits a neighbourhood $U$ such that on any connected component $\tilde U$ of $\pi^{-1}(U)$, $p$ admits a unique antecedant and $\pi\colon \tilde U \to U$ is surjective. We endow $\tilde X$ with the sheaf $\mathcal O _{\tilde X}$ where $\mathcal O_{\tilde X}(U)$ consists of regular functions $f\in \mathcal O _{\tilde M}(\tilde M \cap U)$ that are bounded on any compact subset $K\subset U$. We clearly have $\mathcal O _{\tilde X \mid \tilde M} = \mathcal O _{\tilde M}$. We have to verify that $\mathcal O _{\tilde X}$ gives $\tilde X$ the structure of a Riemann surface. Let $q\in \tilde X \setminus \tilde M$ and $p= \pi(q)$, by hypothesis the meridian loop $\gamma$  around $p$ has a finite order $n_\gamma$ in $\Gamma$ ; it follows that $p$ admits a neigbourhood $U$ biholomorphic to the disk $\Delta$, the connected component $\tilde U$ that contains $q$ is homeomorphic to a disk $\Delta$ and in those coordinates $\pi$ is given by
\[
\begin{aligned}
\pi\colon &\tilde U &\to&\, U \\
& z &\mapsto& z^{n_\gamma +1}.
\end{aligned}
\]
It follows that $\pi\colon \tilde U \to U$ is proper, and if $f$ is an holomorphic function on $U\setminus\{q\}$, it can be written in coordinates $f(z)= \underset{n\in \Z}{\sum} a_n z^n$, it is bounded near $q$ if and only if $a_n = 0$ when $n<0$, it follows that $\mathcal O _{\tilde X}$ gives $\tilde X$ the structure of a Riemann surface. The ramification set $R_\pi$ and its image are included in $\tilde X \setminus \tilde M$ and $\Sigma$ respectively so $\pi\colon \tilde X \to X$ is a ramified covering.
\end{proof}

\begin{rem}
In the case of the theorem, take $q\in \tilde X \setminus \tilde M$, we have a distinguished neighbourhood $\tilde U\simeq \Delta$ of $q$, such that
\[
\begin{aligned}
\pi \colon &\tilde U \simeq \Delta &\to \pi(\tilde U) &:= U\simeq \Delta \\
	     & z &\mapsto z^{n_{\tilde U}}
\end{aligned}
\]
If $n$ is the order of the meridian loop around $p$ in $\Gamma$, one has $n_{\tilde U} = n+1$ so $n_{\tilde U}$ only depends on $\pi(q)$.
\end{rem}

\section{The complexes of sheaves locally square integrable forms}

Again in this section we consider an embedding $j\colon M \to X$, where $M$ is an open Riemann surface, $X$ is a compact Riemann surface and such that $\Sigma := X \setminus M$ is a finite set of points. We consider a Galois covering $\pi \colon \tilde M \to M$, we assume that $\tilde M$ is embedded in a Riemann surface $\tilde X$ and that $\pi$ extend to a ramified covering $\pi \colon \tilde X \to X$. We consider a polarized complex variation of Hodge structure $(\V,F^\bullet,\bar F^\bullet,S)$ on $M$, we will set $n = rk(\V)$. 

We introduce here some complexes of sheaves of locally square integrable forms. All the sheaves introduced in this section will be defined on $X$.

\subsection*{The sheaf $\ell^2\pi_*\V$}

The sheaf $\pi_!\pi^*\V$ on $X$ is a locally constant sheaf of $\C[\Gamma]$-modules, where the $\C[\Gamma]$-module structure comes from by the left regular action on the space of locally constant function on $\pi^{-1}(U)$ i.e $\gamma\cdot f = f\circ \gamma^{-1}$. If we consider the constant sheaf $\ell^2(\Gamma)_M$ on $M$, it has naturally a structure of $\C[\Gamma]$-bimodule, and we define the sheaf 
\[
\ell^2\pi_*\V:= j_*\left(\ell^2(\Gamma)_M \otimes_{\C[\Gamma]} \pi_!\pi^*\V\right).
\]

In \cite{Eyssidieux22}, given a covering Galois $\pi\colon \tilde X \to X$ with covering group $\Gamma$, a functor $\ell^2\pi_*$ is defined on the category of constructible sheaves on $X$ to the category of weakly constructible sheaves on $X$ by 
\[
\ell^2\pi_*F := \ell^2(\Gamma)_X \otimes_{\C[\Gamma]} \pi_!\pi^F.
\]
To show the consistency of our notation with the one in \cite{Eyssidieux22}, we will show the following.

\begin{prop}
There exists a natural isomorphism
\[
\ell^2(\Gamma)_X \otimes_{\C[\Gamma]} \pi_!\pi^*j_*\V \to j_*\left(\ell^2(\Gamma)_M \otimes_{\C[\Gamma]} \pi_!\pi^*\V \right).
\]
is an isomorphism.
\end{prop}
\begin{proof}
There is a natural morphism of sheaves
\[
j_*\ell^2(\Gamma)_M \otimes_{\C[\Gamma]} j_*\pi_!\pi^*\V \to j_*\left(\ell^2(\Gamma)_M \otimes_{\C[\Gamma]} \pi_!\pi^*\V \right)
\]
It is an isomorphism since the natural morphism of functor $j^{-1}j_* \to id$ is an isomorphism in this case. So one has to check that the left side is isomorphic to $\ell^2(\Gamma)_X \otimes \pi_!\pi^* j_*\V$. It is clear that the left side is isomorphic to $\ell^2(\Gamma)_X \otimes j_*\pi_!\pi^* \V$. We begin to show 
\[
j_*\pi_!\pi^* =  \pi_!\tilde j_*\pi^*
\]
where $\tilde j \colon \tilde M \to \tilde X$ is the inclusion. Indeed if $F$ is a sheaf on $M$ one has $j_*\pi_!\pi^*F =  \pi_!\tilde j_*\pi^*F$ on $M$, and if $p\in \Sigma$ is a puncture we take an open neigbourhood $U$ of $p$ with $U\cap M$ biholomorphic to $\Delta^*$. Set $\tilde U$ a connected component of $\pi^{-1}(U\cap M)$ so $\pi^{-1}(U\cap M)$ is biholomorphic to $\Gamma/<\gamma> \times \tilde U$ where $\gamma$ is the meridian loop around $p$. With this identification one has 
\[
\begin{aligned}
\pi_! \tilde j_*\pi^*F(U) &= \{ s \in \Gamma(\pi^{-1}(U)\cap \tilde M,\pi^*F) \mid \pi\colon \tilde X \to X \text{ is proper on } supp(s) \} \\
				  &\simeq \{ s \in \Gamma(\pi^{-1}(U\cap M),\pi^*F) \mid \pi\colon \tilde M \to M \text{ is proper on } supp(s) \} \\
				  &= j_* \pi_! \pi^* F.
\end{aligned}
\]
Now the result follows from $j \circ \pi = \pi \circ \tilde j$ which implies $\tilde j_* \pi^* = \pi^*j_*$.
\end{proof}

\subsection*{The complex $\Omega^\bullet(\pi^*\V)_{(2)}$}

The complex $\Omega^\bullet(\pi^*\V)_{(2)}$ is the complex 
\[
\begin{tikzcd}
\mathcal O (\pi^*\V)_{(2)} \arrow[r, "\nabla"] & \Omega^1(\pi^*\V)_{(2)}
\end{tikzcd}
\]
where $\mathcal O (\pi^*\V)_{(2)}(U)$ (resp. $\Omega^1(\pi^*\V)_{(2)}(U)$) is the space of holomorphic functions (resp. holomorphic $1$-forms) on $\pi^{-1}(U\cap M)$ that are square integrable on every compact $\pi^{-1}(K\cap M)$ for $K$ compact of $U$. In case of the trivial covering, we have seen that the square integrability near the punctures can be interpreted in terms of the monodromy filtration $W_\bullet$ and the Deligne extension. So it will be useful to characterize those on the pull back of the variation of Hodge structure $(\pi^*\V, \pi^*F^\bullet, \pi^*\bar F ^\bullet, \pi^*h)$.

Fix $U$ a neighbourhood of a puncture $p\in X \setminus M$ biholomorphic to a disk $\Delta$. Recall that if $T$ denotes the local monodromy near a puncture one has a decomposition $\V_{|U} = \bigoplus_{\alpha \in ]-1,0]} \V_\alpha$, where on $\V_\alpha$ the operator $T$ has the form $e^{2i\pi\alpha}e^{N_\alpha}$ and $N_\alpha$ is nilpotent and induces an increasing filtration $W_\bullet(\V_\alpha)$ on $\V_\alpha$ and we defined $W_\bullet(\V) = \bigoplus W_\bullet(\V_\alpha)$. Take $\tilde U$ a connected component of $\pi^{-1}(U\cap M)$, on $\tilde U$ one has a decomposition
\[
\pi^*\V(\tilde U) = \bigoplus_{\beta \in ]-1,0]} \big(\pi^*\V\big)_\beta.
\]
Since the monodromy of $\pi^*\V$ on $\tilde U$ is given by $T^{n_p}$ one has an identification
\[
\big(\pi^*\V\big)_\beta = \bigoplus_{n_p\alpha = \beta mod \,1} \pi^*(\V_\alpha).
\]
On $\pi^*\V_\alpha$ the monodromy is given by $e^{2i\pi[n_p\alpha]}e^{n_pN_\alpha}$ where $n_p$ is the order of the stabilizer of $\tilde U$, $[n_p\alpha]$ denotes the representant of $n_p\alpha\, mod\, 1$ in $]-1,0]$. The filtration on $(\pi^*\V)_\beta$ is thus induced by the nilpotent operator $\tilde N _\beta = \bigoplus n_p N_\alpha$. From this we deduce the following lemma.

\begin{lem} 
One has
\[
W_\bullet((\pi^*\V_{|\tilde U})_\beta) = \bigoplus_{[n_p \alpha] = \beta} \pi^*W_\bullet(\V_\alpha) \qquad W_\bullet(\pi^*\V_{|\tilde U}) = \pi^*W_\bullet(\V)
\]
\end{lem}

\subsection*{The complex $\mathcal L ^2 DR^\bullet(\pi^*\V)$}

For an open $U\subset X$, the complex $(\mathcal L^2 DR^\bullet(\pi^*\V)(U),D)$ is the complex of $\pi^*\V$-valued measurable $k$-forms $\phi$ over $\pi^{-1}(U\cap M)$ such that $\phi$ and $D\phi$ are square integrable on $\pi^{-1}(K\cap M)$ for all compact $K \subset U$ (here $D$ is computed in the sense of distribution). It is clear that the presheaves
\[
U \mapsto \mathcal L^2 DR^k(\pi^*\V)(U)
\]
are actually sheaves. Since a smooth form on $X$ has bounded norm for the metric $\omega_{Pc}$, if $\chi$ is a smooth function, and $\phi \in  \mathcal L^2 DR^k(\pi^*\V)$ the form
\[
\chi \cdot \phi = (\chi \circ \pi)\phi
\]
is also in $\mathcal L^2DR^k(\pi^*\V)$ and there is a natural structure of $C^\infty_X$-module on the sheaves $\mathcal L^2DR^k(\pi^*\V)$. It implies the following.
\begin{lem}
The complex $\mathcal L^2DR^\bullet(\pi^*\V)$ is a complex of fine sheaves, and the following equality holds.
\[
\mathbb H^\bullet (X, \mathcal L^2 DR^\bullet(\pi^*\V)) = H^\bullet_2(\tilde M,\pi^*\V, \omega_{Pc}, \pi^*h).
\]
\end{lem}

\subsection*{The complex $\mathcal L^2DR_\infty^\bullet(\pi^*\V)$}

For each open $U\subset X$, we can consider the Laplace operator $\Delta_D = (D+\mathfrak d )^2$ computed in the sense of distribution. We set $\mathcal L^2DR^k_0(\pi^*\V) = \mathcal L^2DR^k(\pi^*\V)$ and as for the global $L^2$ De Rham complex of $\S1$, we can define the sheaves $\mathcal L^2DR^k_j(\pi^*\V)$ by induction, where $\phi \in \mathcal L^2DR^k_{j+1}(\pi^*\V)(U)$ if and only if $\phi \in \mathcal L^2DR^k_{j}(\pi^*\V)(U)$ and $D \Delta^j\phi$ and $\mathfrak d \Delta^j \phi$ are square integrable on $\pi^{-1}(K\cap U)$ for any $K \subset U$ compact. We define
\[
\mathcal L^2DR^k_\infty(\pi^*\V) := \bigcap_{j\in \N} \mathcal L^2DR^k_{j}(\pi^*\V)
\]
which gives a subcomplex of sheaves $\mathcal L^2DR_\infty^\bullet(\pi^*\V)$. It is not clear that it is a sheaf of $C^\infty_X$ module in general since the covariant derivatives of the metric $\omega$ are involved in the computation of the adjoint $\mathfrak d$. However one can still have the following result.

\begin{lem}
The complex $\mathcal L ^2DR^\bullet_\infty$ is a complex of soft sheaves. And one has
\[
\mathbb H^\bullet(X,\mathcal L ^2DR^\bullet_\infty) = H^\bullet_{(2)}(\tilde M,\pi^*\omega_{Pc},\pi^*\V, \pi^*h).
\]
\end{lem}
\begin{proof}
The last assertion of the lemma comes from the fact that soft sheaves are acyclic, so its hypercohomology coincides with the cohomology of the complex $\Gamma \mathcal L ^2DR_\infty^\bullet(\pi^*\V) = L^2DR_\infty^\bullet(\tilde M, \pi^*\V)$, and the conclusion follows the Lemma \ref{quasi-iso}.

To prove the first part, by \cite[Theorem 3.4.1]{Godement} we need to check that every point $p\in X$ admits an open neighbourhood $U$ such that if $K\subset U$, with $K$ closed in $X$, then the restriction of global sections to $K$ is surjective. Take $U$ a neighbourhood of a puncture, with $U\cap M \simeq \Delta_{r}^*$, and $K \subset U$ a closed subset of $M$, and $\phi \in \mathcal L ^2DR^k(\pi^*\V)(K)$. The form $\phi$ is defined on a open neighbourhood $W$ of $K$ that we can assume to be relatively compact on $U$, one can take a smooth function $\chi$, where $\chi = 1$ on a neighbourhood of $K$ and $supp(\chi) \subset W$, then $(\chi\circ \pi)\phi$ is a form that extends $\phi$ to $\pi^{-1}(U)$ and it is also  an element of $\mathcal L^2DR^k_\infty(\pi^*\V)(U)$. Thus the sheaf  $\mathcal L^2DR^k(\pi^*\V)$ is soft.
\end{proof}

\section{The $L^2$ Holomorphic Poincaré Lemma}

This section is dedicated of proving the following result which is an equivariant version of Zucker Poincaré lemma \cite[Proposition 4.1]{Zucker79}.

\begin{thm}[Holomorphic Poincaré lemma]
The following sequence of sheaves is exact.
\[
\begin{tikzcd}
0 \arrow[r] & \ell^2\pi^*\V  \arrow[r] &  \mathcal O (\pi^*\V)_{(2)} \arrow[r, "\nabla"] & \Omega^1(\pi^*\V)_{(2)} \arrow[r] & 0
\end{tikzcd}
\]
\end{thm}
\begin{proof}
The exactness of the sequence
\[
\begin{tikzcd}
0\arrow[r] & \ell^2\pi^*\V \arrow[r]& \mathcal O(\pi^*\V)_{(2)} \arrow[r, "\nabla"] & \Omega^1(\pi^*\V)_{(2)}
\end{tikzcd}
\]
is clear, so we only need to check the surjectivity of
\[
\nabla\colon \mathcal O(\pi^*\V)_{(2)} \to \Omega^1(\pi^*\V)_{(2)}.
\]

We begin with case of points of $p\in M$. We can take a relatively compact neighbourhood $U\subset M$ quasi-isometric to a disk $\Delta$ endowed with the euclidean metric. Fix $\tilde U$ a connected component of $\pi^{-1}(U)$, if $\gamma \in \Gamma$, it defines an isometry between $\tilde U$ and $\gamma \tilde U$. Fix $(\tilde e_1, \dots \tilde e_n)$ a horizontal frame of $\tilde U$, then $(\gamma^*\tilde e_1, \dots \gamma^*\tilde e_n)$ is also an horizontal frame of $\gamma\tilde U$. There exists a constant $C_0>1$ satisfying for all \[
C_0^{-1} < \|\tilde e_i\|_\infty = \|\gamma^* \tilde e_i\|_\infty < C_0
\]
where $\|\cdot\|_\infty$ denotes the supremum of the Hodge norm. We extends the $\tilde e _i$ to $\pi^{-1}(U)$ by $0$, and it follows that the $(\gamma\tilde e _i)_{1\leq i \leq n,\, \gamma \in \Gamma}$ defines an $L^2$-adapted frame. Moreover $\tilde U$ is quasi-isometric to $\Delta$, one can fix such quasi-isometry $\phi$, so there is a positive $C_1$ such that if $\omega$ is a form on $\tilde U$
\[
C_1^{-1} \|\omega\|_2^2 \leq \|\phi^*\omega\|_2^2 \leq C_1 \|\omega\|_2^2.
\]
Since $\tilde U$ is isometric to $\gamma \tilde U$, the later is also quasi-isometric to the disk $\Delta$ with the same Lipschitz constant $C_1$. Take $\eta \in \Omega^1(\pi^*\V)_{(2)}(U)$ given in coordinates by $\eta := \sum_{\underset{1\leq i\leq n}{\gamma\in \Gamma}} f_{i,\gamma} dz \otimes \gamma^* e_i \in \Omega^1(\pi^{-1}\V)_{(2)}(U)$ one has by the previous discussion
\[
C^{-1}\sum_{\underset{1\leq i\leq n}{\gamma\in \Gamma}} \|f_{i,\gamma}\|^2_{2,\Delta} \leq \|\eta\|^2_{2} \leq  C \sum_{\underset{1\leq i\leq n}{\gamma\in \Gamma}} \|f_{i,\gamma}\|^{2}_{2,\Delta} 
\]
where $C$ is a positive constant and the subscript $\Delta$ is here to emphasize the fact that we consider the $L^2$-norm induced by the Euclidean metric on the disk. We are thus reduced to the problem of $L^2$ estimates on the disk, this is given by the Wirtinger inequality: there exists a constant $K> 0$ if $g \in L^2(\Delta)$ satisfies $\int_\Delta g dLeb = 0$ and $dg$ is a square integrable form then
\[
\|g\|^2 \leq K \|dg\|^2.
\]
We take $F_{i,\gamma}$ a primitive of $f_\gamma$ and set $g_{i,\gamma} = F_\gamma - \frac{1}{\pi}\int_{\Delta} F_\gamma(z) dLeb(z)$ one can set 
\[
\nu := \sum_{\underset{1\leq i\leq n}{\gamma\in \Gamma}} g_{i,\gamma} \otimes \gamma^* e_i.
\]
Then we have $\nabla \nu = \eta$ and $\|\nu\|_2^2 \leq KC\|\eta\|_2^2$ thus it is square integrable.

The case of punctures $p\in X\setminus M$ will be similar to the proof of \cite[Proposition 3.12]{Zucker76}, we fix such a point $p$, $U$ a distinguished neighbourhood of $p$. Fix $\tilde U$ a connected component of $\pi^{-1}(U\cap M)$ it is quasi-isometric to some punctured disk $\Delta_r^*$ endowed with the Poincaré metric. There is a decomposition $(\pi^*\V)_{\tilde U} = \underset{\beta \in ]-1,0]}{\bigoplus} (\pi^*\V)_{\beta}$, by taking a multivalued horizontal frame on each $(\pi^*\V)_{\beta}$ flagged according to the monodromy filtration $W_\bullet((\pi^*\V)_\beta)$, we obtain a multivalued horizontal frame $(\tilde e _1, \dots, \tilde e_n)$ of $(\pi^*\V)_{|\tilde U}$. If $\tilde e _i \in (\pi^*\V)_\beta$ we set
\[
\tilde \xi_i := z^\beta exp\left(\frac{\tilde N_\beta}{2i\pi}z\right) \tilde e_i \in (\mathcal V_*^{>-1})_{\tilde U}.
\]
The $(\tilde \xi _1,\dots,\tilde \xi _n)$ defines a $L^2$-adapted frame of $\bar{\mathcal V}$. Same as before if $\gamma \in \Gamma$, then $(\gamma^*\tilde \xi _1,\dots,\gamma^*\tilde \xi _n)$ is an $L^2$-adapted frame of $\gamma \tilde U$. Set $H := Stab_\Gamma(\tilde U)$, and fix a representant $\gamma$ in each class of $\Gamma/H$.

Let $\eta = \sum_{[\gamma] \in \Gamma / H} \eta_\gamma$ be of square integrable holomorphic $1$-form on $\pi^{-1}(U\cap M)$, where each $\eta_\gamma$ has the form
\[
\eta_\gamma := \sum_{k=1}^n f_{k,\gamma}dz\otimes \gamma^*\tilde \xi_k
\]
And there exists a positive constant $C_0$ independant of $g$ satisfying
\[
C_0^{-1} \sum_{k=1}^n \|f_{k,\gamma}dz\otimes \gamma^* \tilde \xi_k\|^2\leq \|\eta_\gamma\|_2^2 \leq C_1 \sum_{[\gamma] \in \Gamma/H} \sum_{k=1}^n \|f_{k,\gamma}dz\otimes \gamma^* \tilde \xi_k\|^2.
\]
Since all $\gamma \tilde U$ are quasi-isometric to the punctured disk $\Delta_r^*$ with the same Lipschitz constant, it follows that there exists a positive constant $C_1$ satisfying
\[
\begin{aligned}
C_1^{-1} \sum_{[\gamma] \in \Gamma / H}\sum_{k=1}^n \|f_{k,\gamma}dz\otimes \gamma^* \tilde \xi_k\|^2\leq \|\eta\|_2^2 \leq C_0 \sum_{k=1}^n \|f_{k,\gamma}dz\otimes \gamma^* \tilde \xi_k\|^2.
\end{aligned}
\]
As before it is sufficient to find a primitive to each $f_{k,\gamma}dz\otimes \gamma^* \tilde \xi_k$ whose norm is controlled by the norm of $f_{k,\gamma}dz\otimes \gamma^* \tilde \xi_k$. Assume $f_{k,\gamma}dz$ extends as a holomorphic section on the whole disk, then we can set
\[
f_{k,\gamma} := \sum_{m=0} ^{+\infty} a_{m} z^m 
\]
and set $\beta \in ]-1,0]$ such that $\tilde e_k$ is a multivalued horizontal section of $(\pi^*\V)_\beta$
\[
\nu_{k,\gamma} :=  \sum_{j=0}^n (-1)^j \sum_{m=0}^{+\infty} a_m(\beta+ m+1)^{j-1}z^{m+1 + \beta} \exp\left(\frac{\tilde N _\beta}{2i\pi}\ln(z)\right)\left(\frac{\tilde N _\beta}{2i\pi}\right)^j \gamma^*\tilde e_k
\]
The term $(\beta+ m+1)^{j-1}$ make sense for any values of $j,m$ since $\beta > -1$, and $\left(\frac{\tilde N _\beta}{2i\pi}\right)^j \gamma^*\tilde e_k $ is a horizontal multivalued section of $(\pi^*\V)_{\gamma \tilde U}$, it follows that the
\[
z^\beta \exp\left(\frac{\tilde N _\beta}{2i\pi}\ln(z)\right)\left(\frac{\tilde N _\beta}{2i\pi}\right)^j \gamma^*\tilde e_k
\]
is a well defined holomorphic section the Deligne extension, it follows that $\nu_{k,\gamma}$ is well defined. One can check
\[
\begin{aligned}
\nabla(\nu_{k,\gamma}) &= \sum_{j=0}^n (-1)^j \sum_{m=0}^{+\infty} a_m( \beta + m+1)^{j}z^{m + \beta}\exp\left(\frac{\tilde N_\beta}{2i\pi}\ln(z)\right)\left( \frac{\tilde N_\beta}{2i\pi}\right)^j \gamma^* \tilde e_k dz \\ 
				&+ \sum_{j=0}^n (-1)^j \sum_{m=0}^{+\infty} a_m(\beta + m+1)^{j-1}z^{m + \beta+1} \exp\left(\frac{\tilde N_\beta }{2i\pi}\ln(z)\right)\!\left(\frac{\tilde N_\beta}{2i\pi}\right)^{j+1}\!\gamma^* \tilde e_k \frac{dz}{z}\\
			&= \sum_{j=0}^n (-1)^j \sum_{m=0}^{+\infty} a_m(\beta + m+1)^{j}z^{m + \beta} \exp\left(\frac{\tilde N_\beta}{2i\pi}\ln(z)\right)\left( \frac{\tilde N_\beta}{2i\pi}\right)^j \gamma^*\tilde e_kdz  \\ 
				&- \sum_{j=1}^{n+1} (-1)^j \sum_{m=0}^{+\infty} a_m(\beta + m+1)^{j}z^{m + \beta} \exp\left(\frac{\tilde N_\beta}{2i\pi}\ln(z)\right)\left(\frac{\tilde N _\beta}{2i\pi}\right)^{j} \gamma^*\tilde e_k dz\\
			&= \sum_{m=0}^{+\infty} a_m z^{m} \gamma^* \tilde \xi_k = f_{k,\gamma}dz\otimes \gamma^*\tilde \xi_k
\end{aligned}
\]
To conclude the proof we only need to see that one can bound the $L^2$ norm of $\nu_{k,\gamma}$ on a punctured disk $\Delta_{r'}^*$ in terms of the $L^2$-norm of $f_{k,\gamma}\otimes \gamma^* \tilde \xi_k$ for any $R'<R<1$. Indeed if $\tilde e_k \in W_l((\pi^*\V)_\beta) \setminus W_{l-1}((\pi^*\V)_\beta)$ there exists a positive constant $C$ that does not depends on $\gamma$ satisfying
\[
\begin{aligned}
C^{-1} \int^R_0 \sum_{m=0}^{+\infty} |a_m|^2r^{2m+ 2\beta}|\ln r|^{l+2}rdr &\leq \|f_{k,\gamma}\otimes \tilde \xi_k\|_{\Delta_R^*}^2 \\
														&\leq C  \int^R_0 \sum_{m=0}^{+\infty} |a_m|^2r^{2m+ 2\beta}|\ln r|^{l+2}rdr.
\end{aligned}
\]
Since the norm of each terms $exp\left(\frac{n_pN}{2i\pi}\ln(z)\right)\left(\frac{n_pN}{2i\pi} \right)^k\tilde w_j$ behave asymptotically as $\ln(r)^{l+2-2k}$ we can bound each term uniformly in $j$ and $\gamma$
\[
\left\|z^\beta exp\left(\frac{\tilde N_\beta}{2i\pi}\ln(z)\right)\left(\frac{\tilde N_\beta}{2i\pi} \right)^j\tilde e_k\right\|^2 \leq C' r^{2\beta}\ln(r)^{l+2}
\]
So one has
\[
\begin{aligned}
\|\nu_{k,\gamma}\|^2_{\Delta_R^*} &\leq C' \sum_{j=0}^{n} \int^{R'}_0 \sum_{m=0}^{+\infty} |a_m|^2(m+1)^kr^{2m+2\beta + 2}|\ln|z||^{l}\frac{dr}{r} \\
	      &\leq nC' \int^{R'}_0 \sum_{m=0}^{+\infty} |a_m|^2(m+1)^nr^{2m+2}|\ln|z||^{l}dr  \\
	      &\leq C' \int^{R}_0 \sum_{m=0}^{+\infty} |a_m|^2(m+1)^n\left(\frac{R'}{R}\right)^{2m+2}r^{2m+1}|\ln|z||^{l}dr \\
	      &\leq C'' \|f_{k,\gamma} dz\otimes \gamma^*\tilde \xi_k\|_{\Delta_R^*}^2.
\end{aligned}
\]
The constant $C''$ being indepandant of $\gamma$ and $f_{k,\gamma}$, this conclude the case where the $f_{k,\gamma}$ are defined on the whole disk. Otherwise by Proposition \ref{caracl2}, since $f_{k,\gamma}\otimes \gamma^*\tilde \xi _k$ is square integrable one has $\tilde \xi_k \in W_{-2}((\pi^*\V)_0)$ and $f_{k,\gamma}$ admits a simple pole at the origin. However $\tilde N_0$ induces a surjection $W_0((\pi^*\V)_0) \to W_{-2}((\pi^*\V)_0)$ thus there exists $\tilde e \in W_0((\pi^*\V)_0)$ with $\tilde N_0 \tilde e = \tilde e_k$, if we set $\tilde\xi = exp\left(\frac{\tilde N_0}{2i\pi}\ln(z)\right) \tilde e$ and $a_{-1}$ the residue of $f_{k,\gamma}$ at the origin then
\[
\|a_{-1}\gamma^*\tilde\xi\|^2 \leq C \|f_{k,\gamma}dz\otimes \gamma^*\tilde \xi_k\|^2
\]
and $f_{k,\gamma}dz\otimes \gamma^*\tilde \xi_k\ - \nabla (a_{-1} \gamma^*\tilde \xi_k)$ is of the previous forms which concludes the proof.
\end{proof}

\section{The $L^2$ Poincaré Lemma}
This section is dedicated to some $L^2$ Poincaré lemmas. Before stating the main result we give the following lemma.
\begin{lem}\label{quasi-iso smooth subcomplex}
The inclusion
\[
\mathcal L ^2 DR_\infty^\bullet(\pi^*\V) \to \mathcal L ^2DR^\bullet(\pi^*\V)
\]
is a quasi-isomorphism.
\end{lem}

\begin{proof}
For a point $p\in X$ we will denote by $\mathcal H ^k_p$ (resp. $\mathcal H ^k_{\infty,p}$) the cohomology germs of the complex $\mathcal L ^2DR^\bullet(\pi^*\V)$ (resp. $\mathcal L ^2DR_\infty^\bullet(\pi^*\V)$). If $p\in M$, it is straight forward to see that both $\mathcal H ^k_p$ and $\mathcal H ^k_{\infty,p}$ vanishes if $k\neq0$ and are isomorphic to $\C$ if $k=0$. So the we will focus on the case of $p\in X\setminus M$.

Let $\phi$ is the germ of a closed form in $\mathcal L ^2DR_\infty^\bullet(\pi^*\V)_p$, if there exists $\psi \in \mathcal L ^2DR^\bullet(\pi^*\V)_p$ such that 
$D \psi = \phi$, one can take a neighbourhood $U$ of $p$, such that $\phi$ and $\psi$ are square integrable on $\tilde U := \pi^{-1}(U\cap M)$. With the notation of section $\S1$ we can view $\phi$ and $\psi$ as element of the Hilbert space $L^2DR^\bullet(\tilde U, \pi^*\V)$, satisfying $\phi = D_{max} \psi$. By the weak Hodge decomposition
\[
L^2DR^\bullet(\tilde U, \pi^*\V) = Harm \oplus \overline{Im(D_{max})} \oplus \overline{Im(\mathfrak d _{min})}
\]
where $Harm$ is the space of forms that are $D_{max}$-closed and $\mathfrak d _{min}$-closed. One can set $\psi_0 = pr_{\overline{Im(\mathfrak d _{min})}} \psi$ which is square integrable on $\tilde U$ and satisfies $D_{max}\psi_0 = D_{max}\psi = \phi$. We set  the operator $\square_{D_{max}} =(D_{max}+\mathfrak d _{min})^2 $ and for all $j\in \N$ one has
\[
\square^{j+1}_{D_{max}}\psi_0 = \mathfrak d _{min}\square_{D_{max}}^j D_{max}\psi_0 = \mathfrak d_{min} \square_D^j \phi
\]
so $\square_{D_{max}}^j\psi_0$ is defined for all $j\in \N$, noting that $\square_{D_{max}}$ is equal to $\Delta_D$ where it is defined, one obtains that $\psi_0 \in \mathcal L ^2DR_\infty^\bullet(\pi^*\V)_p$ and the inclusion induces an injection in cohomology.

For the surjectivity, if $\phi \in \mathcal L ^2DR^\bullet(\pi^*\V)_p$ is a closed form, we fix $U$ neighbourhood of $p$ such that $\phi$ is a closed a square integrable form on $\tilde U := \pi^{-1}(U\cap M)$, by \cite[Theorem 2.12]{Bruning}, there exists a closed form $\psi$ also square integrable on $\tilde U$ in the cohomology class of $\phi$, that satisfies
\[
\psi \in \bigcap_{j\in \N} Dom(\square^j_{D_{max}})
\]
and again since $\square_{D_{max}}$ coincides with $\Delta_D$ where it is defined, one obtains the surjectivity.
\end{proof}

The main result of this section is the following theorem. It is an extension of Zucker's Poincaré lemma \cite[Theorem 6.2]{Zucker79} to an equivariant setting.

\begin{thm}[$L^2$ Poincaré lemma]\label{Pclem}
The complexes of sheaves $\mathcal L ^2DR^\bullet(\pi^*\V)$ and $\mathcal L ^2DR^\bullet_\infty(\pi^*\V)$ are soft resolutions of $\ell^2\pi_*\V$.
\end{thm}
Since the inclusion of $\mathcal L^2DR_\infty(\pi^*\V) \to \mathcal L^2DR^\bullet(\pi^*\V)$ we only need to prove it for the case of the complex $\mathcal L ^2DR^\bullet(\pi^*\V)$.

By definition, it suffices to check that for all $p\in X$, $k\geq 0$
\[
\bigslant{\mathcal L^2DR^{k+1}(\pi^*\V)(U)}{\mathcal L^2DR^{k}(\pi^*\V)(U)} = 0
\]
for some $U$ defining a fundamental basis of neighbourhood of $p$.

If $U$ is an open of $X$ and $k\in \N$ we will denote by $Dom_{max}^k(\pi^{-1}(U))$ the domain of the operator $D_{max}$ defined on the Hilbert space $L^2DR^k(\pi^{-1}(U), \pi^*\V)$. If $V\Subset U$ is open the restriction maps of the complex of sheaves $\mathcal L^2DR^\bullet(\pi^*\V)$ factorizes
\[
\begin{tikzcd}
\mathcal L^2DR^k(\pi^*\V)(U) \ar[r, "D"] \ar[d]& \mathcal L^2DR^{k+1}(\pi^*\V)(U) \ar[d]\\
Dom_{max}^k(\pi^{-1}(V)) \ar[r, "D_{max}"] \ar[d]& Dom_{max}^{k+1}(\pi^{-1}(V)) \ar[d]\\
\mathcal L^2DR^k(\pi^*\V)(V) \ar[r, "D"] & \mathcal L^2DR^{k+1}(\pi^*\V)(V)
\end{tikzcd}
\]

Hence it suffices to check that for all $p$, that for the ideal boundary condition $D_{max}$
\[
H^k_2(\pi^{-1}(U),\pi^*\V) = 0
\]
for open $U$ in a fundamental basis of neighbourhood of $p$. We denote by $\gamma$ the image the meridian circle around $p$ in $\Gamma$ (which is trivial of $p\in M$) and $\tilde U$ a connected component of $\pi^{-1}U$, for $U$ small enough $\pi^{-1}(U)$ is isometric to $\Gamma/<\gamma>\times \tilde U$ hence one has an isomorphism of Hilbert spaces
\begin{equation}\label{isomorphism}
L^2DR^k(\pi^{-1}(U),\pi^*\V) \simeq \ell^2(\Gamma/<\gamma>)\hat{\otimes}L^2DR^{k+1}(\tilde U, \pi^*\V)
\end{equation}
where $\hat{\otimes}$ denotes the Hilbert tensor product. Moreover consider the densely defined unbounded operator $id\otimes D_{max}$ with domain
\[
\begin{aligned}
Dom(id\otimes D_{max}) &= \ell^2(\Gamma/<\gamma>) \otimes Dom(D_{max})\\
				       &\subset \ell^2(\Gamma/<\gamma>) \otimes L^2DR^{k+1}(\tilde U, \pi^*\V)\\
				       & \subset \ell^2(\Gamma/<\gamma>)\hat{\otimes}L^2DR^{k+1}(\tilde U, \pi^*\V)
\end{aligned}
\]
where $\otimes$ is the usual tensor product. It is closable since $id$ and $D_{max}$ are closed and it is straight forward that its minimal closure $id\hat{\otimes}D_{max}$ coincides via the isomorphism (\ref{isomorphism}) to the closed operator $D_{max}$ defined on $L^2DR^k(\pi^{-1}(U),\pi^*\V)$. Similarly it is clear that minimal closure of $id\otimes \mathfrak{d}_{min}$ is the operator $\mathfrak{d}_{min}$ on $L^2DR^k(\Gamma/<\gamma>\times \tilde U, \pi^*\V)$. Hence the Laplace operator $\square_{D_{max}}$ on $L^2DR^k(\Gamma/<\gamma>\times \tilde U, \pi^*\V)$ is given by $id\otimes \square_{D_{max}}$ via this isomorphism and since the spectra of $id\otimes\square_{D_{max}}$ is equal to the spectra of $\square_{D_{max}}$ on $L^2DR^k(\tilde U, \pi^*\V)$ hence by Lemma \ref{criteria of reduced cohomology} we obtain the following result.

\begin{lem}\label{reduction Poincare lemma}
If any point $p\in X$ admits a fundamental basis of open connected neighbourhood $(U_i)_{i\in I}$ such that
\[
H^k_2(\tilde U_i, \pi^*\V,\pi^*\omega_{Pc},h) =0 \qquad \forall k\geq 1
\]
where $\tilde U_i$ is a connected component of $\pi^{-1}U$ then the $L^2$-Poincaré lemma hold.
\end{lem}

For the case of point $p\in M$, it admits a fundamental basis of neighbourhood quasi-isometric to a convex subsbpace $D$ of $\C$, and a connected component of $\pi^{-1}(U)$ will thus be quasi-isometric to a $D$. In this case the result is given by the following lemma.
\begin{lem}\cite[Lemma 4.2]{Iwaniec-Lutoborski}
If $D$ is an bounded convex open subset of $\R^n$ and $\omega$ is a closed square integrable $k$-form, there exists $\nu$ a square integrable $k-1$-form satisfying
\[
\|\nu\|_2 \leq C \|\omega\| \qquad D_{max}\nu = \omega
\]
where $C$ is a constant depending on $n$ and the geometry of $D$. In fact one can take
\[
C = \frac{2^n\sigma_{n-1}Diam(D)^{n+2} }{\int_D dist(x,\del D)dLeb(x)}
\]
where $\sigma_{n-1}$ is the volume of the sphere of dimension $n-1$. 
\end{lem}

For the case punctures $p\in X \setminus M$, $p$ has a fundamental basis of neighbourhood consisting of open $U$ quasi-isometric to a punctured disk $(\Delta_R^*,\omega_{Pc})$, so if $\tilde U$ is a connected component of $\pi^{-1}(U)$ it is quasi-isometric to $(\Delta_{R^{1/k}}^*,\omega_{Pc} )$ where $k$ is the order of the subgroup $<\gamma> \subset \Gamma$. Hence we are reduced to show the following.

\begin{lem}\label{diskcase}
For the ideal boundary condition $D_{max}$, one has
\[
H^k(L^2DR^\bullet(\Delta_R^*, \pi^*\V)) = 0 \qquad \text{ if } k=1,2.
\]
\end{lem}
This is essentially what Zucker showed with his $L^2$ Poincaré lemma \cite[Theorem 6.2, Proposition 11.3]{Zucker79}, albeit it is stated in a slightly weaker version, the proof is essentially the same. We give the computation nonetheless for the sake of completeness. 
The complex $L^2DR^\bullet(\Delta_r^*, \pi^*\V)$ is filtered by the increasing filtration $W_\bullet(\pi^*\V)$, and we will compute the cohomology via a spectral sequence. It is to be noted that since a frame flagged according to $W_\bullet(\pi^*\V)$ is $L^2$ adapted one has a canonical isomorphism
\[
Gr^W L^2DR^\bullet(\Delta_r^*, \pi^*\V)) = L^2DR^\bullet(\Delta_r^*, Gr^W(\pi^*\V))
\]
Fix $(\tilde \xi_1,\dots,\tilde \xi_n)$ a $L^2$ adapted basis obtain via a multivalued horizontal frame $(\tilde e_1,\dots ,\tilde e_n)$ flagged according to the filtration $W_\bullet$ and compatible with the decomposition $\pi^*\V = \bigoplus (\pi^*\V)_\beta$. If $\tilde e_i \in W_l(\pi^*(\V)_\beta) \setminus W_{l-1}(\pi^*(\V)_\beta)$ and $\omega$ is a form one has
\[
D_{max}(\omega \otimes \xi_i) = d\omega\otimes \xi_i + \beta\omega\wedge \frac{dz}{z}\otimes \xi_i \,mod \left( W_{l-1} \mathcal L^2DR^\bullet(\Delta_r^*, \pi^*\V) \right).
\]
We recall that $Gr^W_k(\pi^*\V) = \bigoplus Gr^W_k((\pi^*\V)_\beta)$

The following lemma is a reformulation of \cite[Proposition 6.6, Proposition 11.3]{Zucker79}.
\begin{lem}\label{lemZucker}
We have the following results.
\begin{enumerate}
\item The cohomology group $H^0_2(\Delta_R^*, Gr^W_k(\pi^*\V)_\beta)$ is equal to the space of global sections of  $Gr_k^W((\pi^*\V)_\beta)$ if $k<1$ and $\beta = 0$, it vanishes otherwise.
\item If $\eta \in L^2DR^1(\Delta_R^*, Gr^W((\pi^*\V)_\beta)$ is a closed one form, there exists a $0$-form $\nu$ such that $\|\nu\|_2^2 \leq K \|\eta\|_2^2$ and
\[
\eta-D\nu = \left\{\begin{array}{ll} 0 & \text{ if } \beta \neq 0 \\
						 0 &\text{ if } \beta = 0, \, k\geq-1 \text{ or } k\neq1 \\
						 \phi  & \text{ if } \beta = 0, \, k=1 \\
						 \frac{dz}{z}\otimes e & \text{ if } \beta = 0, \, k<-1
 \end{array}\right.
\] 
where $\phi \in L^2([0,R],ln(r)rdr)dr \otimes Gr^W_1(\pi^*\V)_0$ and $e\in Gr^W_k(\pi^*\V)_0$ is a horizontal section.
\item If $\eta\in L^2DR^2(\Delta_R^*, Gr^W((\pi^*\V)_\beta)$ is a closed two form, there exists a $1$-form $\nu$ such that $\|\eta\|_2^2 \leq K \|\nu\|_2^2$ and
\[
\eta-D\nu = \left\{\begin{array}{ll} 0 & \text{ if } \beta \neq 0 \\
						 0 &\text{ if } \beta = 0, \, k \neq -1 \\
						\psi & \text{ if } \beta = 0, \, k=-1
 \end{array}\right.
\] 
with $\psi \in L^2([0,R],ln(r)rdr)dr\wedge\frac{dz}{z} \otimes Gr^W_{-1}(\pi^*\V)_\beta$.
\end{enumerate}
\end{lem}

For the reader convenience, we reproduce Zucker's arguments.

\begin{proof}
On the bundle $Gr^{W}_k((\pi^*\V)_\beta)$ we choose a $L^2$ adapted frame $(\tilde \xi  _1,\dots, \tilde \xi _k)$, and we consider the trivial action of $\mathbb S ^1$ for this trivialization. This allow us to consider the Fourier coefficients $\eta_n$ of a $Gr^W$-valued form $\eta$ on $\Delta_R^*$, so one has
\[
\eta = \sum_{n\in \Z} \eta_n(r)e^{in\theta}.
\]
Since $D_{max}( \phi\otimes \tilde\xi_i) = d\omega\otimes \tilde\xi_i + \beta\omega\wedge \frac{dz}{z}\otimes \tilde \xi_i$, we can reduce the computation to the case where $\eta$ has the form  $\eta = \phi\otimes \tilde \xi$. We start by setting $\phi$ in polar coordinates : $\phi = fdr + gd\theta$. By decomposing into Fourier series we have
\[
\begin{aligned}
f &= \sum f_n(r) e^{in\theta} \\
g &= \sum g_n(r) e^{in\theta} 
\end{aligned}
\]
The fact that the form is closed implies $g_n' + \frac{\beta}{r}g_n= i(n+\beta)f_n$, for all $n \in \Z$. 
By setting 
\[
h =\left\{ \begin{array}{ll} \sum_{n\in\Z\setminus\{0\}} \frac{g_n}{in}e^{in\theta} & \text{ if } \beta = 0 \\
					\sum_{n\in\Z} \frac{g_n}{i(n+\beta)}e^{in\theta} & \text{ if } \beta \neq 0
		\end{array}
\right.
\] 
we obtain a square integrable section, indeed (for instance if $\beta = 0$)
\[
\begin{aligned}
\int_0^{R} \int_0^{2\pi} |h(r,\theta)|^2 d\theta \frac{\ln(r)^{k-2}}{r}dr &\leq C \int_0^{R} \sum_{n\in \Z\setminus\{0\}} \frac{|g_n(r)|^2}{n^2} d\theta \frac{\ln(r)^{k-2}}{r}dr \\
												      &\leq C \int_0^{R} \sum_{n\in \Z\setminus\{0\}} \frac{|g_n(r)|^2}{n^2} \frac{\ln(r)^k}{r}dr  \\
												      &\leq C\|g\otimes \tilde \xi\|^2\leq C^2\|\eta\|^2
\end{aligned}
\]
It follows that $\nu = h\otimes \tilde\xi$ is square integrable, and that we have
\[
\eta -  D(\nu) = \left \{\begin{array}{ll} f_0(r)dr\otimes \tilde \xi + g_0d\theta\otimes \tilde \xi & \text{ if } \beta = 0 \\
						0 & \text{ if } \beta \neq 0
			\end{array} \right.
\]
This solves the case $\beta \neq 0$. The only problem still arising are in the case $\beta=0$ (i.e in the unipotent part), in this case if $k\neq 1$, the form $f_0(r)dr \otimes \tilde \xi$ is exact as we have $f_0dr\otimes \tilde \xi = D(h\otimes \tilde \xi)$ with
\[
h(r) = \left\{
\begin{array}{cc}
\int_{0}^r f_0(t)dt & \text{ if } k >1 \\
-\int_{r}^R f_0(t)dt & \text{ if } k <1
\end{array}
\right.
\]
And the $0$-form $h\otimes \tilde \xi$ is square integrable. For instance if $k< 1$, we fix $\alpha > 0$ a constant with $k+\alpha < 1$  then
\[
\begin{aligned}
\|h\otimes \tilde \xi\|^2 &\leq C \int_0^R \left| \int_r^R f_0(t)dt \right|^2 \frac{\ln(r)^{k-2}}{r}dr \\
				&\leq C \int_0^R  \int_r^R |f(t)|^2t |\ln t |^{1- \alpha}dt \int_r^R \frac{1}{t |\ln t|^{1-\alpha}}dt \frac{|\ln(r)|^{k-2}}{r}dr \\
				&\leq \frac{C}{\alpha} \int_0^R  \int_r^R |f(t)|^2t |\ln t |^{1- \alpha}dt \int^{r}_{R}\frac{|\ln r|^{k-2+\alpha}}{r}dr \\
				&\leq \frac{C}{\alpha} \int_0^R  \int_r^R |f(t)|^2t |\ln t|^{1- \alpha}dt \int^{r}_{R}\frac{|\ln r |^{k-2+\alpha}}{r}dr \\
				&\leq \frac{C}{\alpha} \int_0^R  |f(t)|^2t |\ln t|^{1- \alpha} \int_0^t \frac{| \ln r |^{k-2+\alpha}}{r}dr dt \\
				&\leq \frac{C}{\alpha} \int_0^R  |f(t)|^2t |\ln t|^{k-1} dt \\
				&\leq \frac{C}{\alpha} \int_0^R  |f(t)|^2t |\ln t|^{k} dt \leq \frac{C^2}{\alpha} \|fdr\otimes \tilde \xi\|^2.
\end{aligned}
\]
The case $k> 1$ is similar.
We still have to deal with forms $g_0 d\theta$, since $g_0$ is a constant so it must vanishes if $k\geq -1$, by the integrability condition. Otherwise one can conclude by remarking $d\theta = \frac{1}{2i}(\frac{dz}{z} - \frac{d\bar z}{\bar z}) = \frac{dz}{iz} - d\ln |z|^2$ and both $\frac{d\bar z}{\bar z}\otimes \tilde \xi$ and $\ln |z|^2\otimes \tilde \xi$ is square integrable if $k<-1$.

The case of two forms is quite similar. If $\eta = h dr\wedge d\theta \otimes \tilde \xi$, then one can write
\[
h = \sum_{n\in \Z} h_n(r)e^{in\theta}
\]
 we can set
\[
\nu = \left\{ \begin{array}{ll}
\sum_{n\in \Z} \frac{1}{\beta - in}h_n(r)e^{in\theta}dr \otimes \tilde \xi & \text{ if } \beta \neq 0 \\
\sum_{n\in \Z\setminus\{0\}} \frac{-1}{in}h_n(r)e^{in\theta}dr \otimes \tilde \xi & \text{ if } \beta = 0 \end{array}
\right.
\]
One has 
\[
\eta - D\nu = \left\{ \begin{array}{ll} -ih_0(r)dr\wedge \frac{dz}{z}\otimes \tilde\xi - id(\ln(|z|^2)h_0(r)dr\otimes \tilde \xi) & \text{ if } \beta = 0 \\
			0 & \text{ if } \beta \neq 0 \end{array}
\right.
\]
and
\[
\begin{aligned}
\|g\otimes \tilde \xi\|_{\Delta_R^*}^2 &= 2\pi\sum_{n\in \Z\setminus \{0\}}\int_{0}^R \frac{|h_n(r)|^2}{n^2}|\ln r|^krdr \\
					    &\leq 2\pi\sum_{n\in \Z\setminus \{ 0\}}\int_{0}^R |h_n(r)|^2|\ln r|^{k+2}rdr\leq  \|\eta\|^2.
\end{aligned}
\]

As before when $k\neq-1$, a the form $h_0(r)dr\wedge d\theta = df\wedge d\theta$ with

\[
f(r) = \left\{
\begin{array}{cc}
\int_{0}^r h_0(t)dt & \text{ if } k >-1 \\
\int_{r}^R h_0(t)dt & \text{ if } k <-1
\end{array}
\right.
\]
 and 
 \[
 \|fd\theta\otimes \tilde \xi\|^2_2 \leq K \|h_0dr\wedge d\theta \otimes \tilde \xi\|^2_2.
 \]
\end{proof}

\begin{proof}[Proof of Lemma \ref{diskcase}.]
We consider the spectral sequence with respect to the filtration $W_\bullet$ on the complex $L^2DR^\bullet(\pi^{-1}(U),\pi^*\V)$. We have
\[
E^{p,q}_1 = H^{p+q}(Gr_W^{-p} L^2DR^\bullet(\pi^{-1}(U),\pi^*\V)) \implies H^{p+q}(L^2DR^\bullet(\pi^{-1}(U),\pi^*\V)).
\]
The differential 
\[
d_1\colon H^{p+q}(Gr_W^{-p} L^2DR^\bullet(\pi^{-1}(U),\pi^*\V)) \to H^{p+q+1}(Gr_W^{-(p+1)} L^2DR^\bullet(\pi^{-1}(U),\pi^*\V))
\] of the spectral sequence vanishes since it is given by $\frac{dz}{2i\pi z}\otimes \tilde N_0$ and 
\[
\tilde N_0 W_k((\pi^*\V)_0) \subset W_{k-2}((\pi^*\V)_0)
\]
 for all $k$. So one has
\[
E^{p,q}_2 = H^{p+q}(Gr_W^{-p} L^2DR^\bullet(\pi^{-1}(U),\pi^*\V)).
\]
with differential
\[
d_2 \colon H^{p+q}(Gr_W^{-p} L^2DR^\bullet(\pi^{-1}(U),\pi^*\V)) \to H^{p+q+1}(Gr_W^{-(p+2)} L^2DR^\bullet(\pi^{-1}(U),\pi^*\V))
\]
which is also induced by $\frac{dz}{2i\pi z}\otimes \tilde N_0$. Noting that $H^{p+q}(Gr_W^{-p} L^2DR^\bullet(\pi^{-1}(U),\pi^*\V)) = 0$ unless $0\leq p+q\leq2$ and $p\geq - 1$ by Lemma \ref{lemZucker}.
We begin to treat the case $p+q = 0$ and $p\geq 0$, where we need to see that $d_2$ is surjective, since $N_0$ induces a surjection $W_{l} \to W_{l-2}$ for any $l\leq 0$.  
It is left to show that
\[
d_2 \colon \colon H^{1}(Gr_W^{1} L^2DR^\bullet(\pi^{-1}(U),\pi^*\V)) \to H^{2}(Gr_W^{-1} L^2DR^\bullet(\pi^{-1}(U),\pi^*\V))
\]
is an isomorphism which come from the fact that $\tilde N _0$ induces an isomorphism between $Gr^W_1((\pi^*\V)_0)$ and $Gr^W_{-1}((\pi^*\V)_0)$. It follows that all the $E_3^{p,q}$ vanishes unless the $E_3^{p,-p}$ with $p\geq 1$. So the cohomology groups $H ^k_2(\Delta^*_R,\pi^*\V)$ vanishes for $k\geq 1$, which conclude the proof of Lemma \ref{diskcase}, and thus the proof of Theorem \ref{Pclem} thanks to the lemma \ref{reduction Poincare lemma}.
\end{proof}

\section{The $L^2$ Dolbeault lemma}

On $M$, we have the filtration $F^\bullet$ of the vector bundle $\mathcal V$, this gives a filtration $\pi^*F^\bullet$ on the bundle $\pi^*\mathcal V$. We obtain a filtration $\pi^*F^\bullet$ on the complex $\Omega^\bullet(\pi^*\V)_{(2)}$ by setting $\pi^*F^p\Omega^k(\pi^*\V)_{(2)} := \Omega^k(\pi^*F^{p-k})_{(2)}$. By $\Omega^k(Gr_{\pi^*F}^p)_{(2)}$ we denote the space of locally square integrable $Gr_{\pi^*\V}^p$-valued $k$-forms. We have a natural morphism $\pi^*F^p\Omega^k(\pi^*\V)_{(2)} \to \Omega^k(Gr^p_{\pi^*F})_{(2)}$, this induce a natural morphism
\[
Gr_{\pi^*F}^p\Omega^k(\pi^*\V)_{(2)} \to \Omega^k(Gr_{\pi^*F}^p)_{(2)}.
\]
It is in fact an isomorphism as stated below.
\begin{prop}
The following sequences are exact
\[
\begin{tikzcd}[column sep = small]
0 \arrow[r] &\pi^* F^{p+1}\mathcal O(\pi^*\V)_{(2)} \arrow[r]& \pi^* F^p\mathcal O (\pi^*\V)_{(2)} \arrow[r]& \mathcal O(Gr^p_{\pi^*F})_{(2)} \ar[r] & 0 \\
0 \arrow[r] &\pi^*F^{p+1}\Omega^1(\pi^*\V)_{(2)}\arrow[r] & \pi^*F^p\Omega^1(\pi^*\V)_{(2)} \arrow[r]& \Omega^1(Gr^p_{\pi^*F})_{(2)} \arrow[r] & 0
\end{tikzcd}
\]
\end{prop}
\begin{proof}
In the case of a trivial covering this is the content of \cite[Proposition 5.2]{Zucker79}. This is done by finding a horizontal basis of $\V$ that respect both the Hodge filtration and the weight filtration, the pull back of such a basis allows us to conclude.
\end{proof}

The two complexes $\mathcal L^2DR_\infty^\bullet (\pi^*\V)$ and $\mathcal L^2DR_\infty^\bullet (\pi^*\V)$ are also filtered by a filtration $\pi^*F$ induced by $F$, where $\pi^*F^P\mathcal L^2DR^k (\pi^*\V)$ consists of element of $\mathcal L^2DR^k (\pi^*\V)$ whose component of type $(r,s)$ is a $\pi^*F^{P-r}$-valued $(r,s)$-form. We define the filtration on $\mathcal L^2DR_\infty^\bullet(\pi^*\V)$ in a similar way. One can define the filtration $\pi^*\bar F$ in the similar way and set
\[
\begin{aligned}
\mathcal L^2DR^\bullet (\pi^*\V)^{P,Q} &= \pi^*F^P \mathcal L^2DR^\bullet (\pi^*\V) \cap \pi^*F^Q \mathcal L^2DR^\bullet (\pi^*\V) \\
\mathcal L^2DR_\infty^\bullet (\pi^*\V)^{P,Q} &= \pi^*F^P \mathcal L^2DR_\infty^\bullet (\pi^*\V) \cap \pi^*F^Q \mathcal L^2DR^\bullet (\pi^*\V)
\end{aligned}
\]
We have a canonical injection
\[
\bigoplus \mathcal L^2DR^\bullet (\pi^*\V)^{P,Q} \hookrightarrow \mathcal L^2 DR^\bullet(\pi^*\V)
\]
but it is not surjective in general. A form $\phi \in \mathcal L^2 DR^\bullet(\pi^*\V)$ can be decomposed as a measurable form
\[
\phi = \sum \phi^{P,Q}.
\]
For $\phi$ to be in $\bigoplus \mathcal L^2DR^\bullet (\pi^*\V)^{P,Q}$, we need that all the $D\phi^{P,Q}$ to be square integrable which is not always the case. The smooth subcomplex has a better behaviour as stated below.
\begin{lem}
The following decomposition holds
\[
\mathcal L^2 DR_\infty^\bullet(\pi^*\V) = \bigoplus \mathcal L^2DR_\infty^\bullet (\pi^*\V)^{P,Q}.
\]
\end{lem}
\begin{proof}
Since the differential operator $\Delta_D$ is elliptic, sections $\phi$ of $\mathcal L^2 DR_\infty^k(\pi^*\V)$ are smooth forms since they belongs in all the local Sobolev spaces. We can write $\phi = \sum \phi^{P,Q}$ where $\phi^{P,Q}$ is its $(P,Q)$ component as a measurable form. From the identity $\Delta_D = \Delta_{D''}$ we know that the Laplace operator preserves the decomposition so for any integer $j\geq 0$ the $\Delta^j_D\phi^{P,Q}$ are locally square integrable and we need to verify that the $D\Delta^j_D\phi^{P,Q}$ are also square integrable. However $D''\Delta^j_D\phi$ is also locally square integrable and by induction on $Q$ it follows that all the $D''\phi^{P,Q}$ are square integrable. We can do the same with $D'$ since $\Delta_{D} = \Delta_{D'}$ and all the $D'\Delta^j_D\phi^{P,Q}$ are locally square integrable, and thus we obtain that $D\Delta^j_D\phi^{P,Q}$ is locally square integrable for all $(P,Q)$.
\end{proof}

We will also consider the complex of sheaves $\mathcal L^2Dolb^{P,\bullet}(\pi^*\V)$ defined on $X$. Sections of $\mathcal L^2Dolb^{P,k}(\pi^*\V)$ over $U\subset X$, are $\pi^*Gr_F^{P-k}$-valued forms such that $\phi$ and $D''\phi$ are square integrable on $\pi^{-1}(K \cap M)$ for any $K$ compact of $U$. The square integrability of $\phi$ and $D''\phi$ is equivalent to the square integrability of $\phi$ and $\bar \del \phi$ since $d''$ induces the $\bar \del$ on the holomorphic bundle $Gr_F^p$ and $\nabla'$ is bounded. From this complex one can construct a subcomplex $\mathcal L^2Dolb_\infty^{P,\bullet}(\pi^*\V)$ of forms $\phi$ such that $\Delta^k_{D''}\phi$ is in $\mathcal L^2Dolb^{P,\bullet}(\pi^*\V)$ for all $k\in \N$. By the existence of the decomposition
\[
\mathcal L^2 DR_\infty^k(\pi^*\V) = \underset{P+Q = w+k}{\bigoplus} \mathcal L^2 DR^k_\infty(\pi^*\V)^{P,Q}.
\]
The projection on the $(P,Q)$ component gives us a natural isomorphism of complex
\[
Gr_{\pi^*F}^P\mathcal L^2DR^\bullet_\infty(\pi^*\V) \to \mathcal L^2Dolb^{P,\bullet}_\infty(\pi^*\V).
\]

\begin{lem}\label{qu-iso Dolbeault}
The inclusion of complexes
\[
\mathcal L^2Dolb^{P,\bullet}_\infty(\pi^*\V) \to \mathcal L^2Dolb^{P,\bullet}(\pi^*\V)
\]
is a quasi-isomorphism.
\end{lem}

\begin{proof}
The proof is similar to the one of the Lemma \ref{quasi-iso smooth subcomplex}. We fix $x\in X\setminus M$ and $\phi$ a germ of closed $(k+1)$-form in $\mathcal L^2Dolb^{P,k+1}_\infty(\pi^*\V)$, if $\phi = D''\psi$ with $\psi \in \mathcal L^2Dolb^{P,k}(\pi^*\V)$. We fix $U$ a neighbourhood of $x$ such that $\phi$, $\psi$ are square integrable on $\pi^{-1}(U\cap M)$. When we view $\phi$ and $\psi$ as element of the Hilbert space $L^2Dolb^{P,\bullet}(\pi^{-1}(U),\pi^*\V)$ we have $D_{max}\psi = \phi$. Using the notation of $\S1$ by the weak Hodge decomposition we have the orthogonal decomposition of the Hilbert space of square integrable forms
\[
L_2(\pi^{-1}(U),\pi^*\V) = Harm_{max} \oplus \overline{Im(D''_{max})} \oplus \overline{Im(\mathfrak d ''_{min})}
\]
since $\mathfrak d''_{min}$ is the adjoint of $D''_{max}$ and we can set $\psi_0 = pr_{\overline{Im(\mathfrak d ''_{min})}} \psi$, and $D''_{max}\psi_0 = \phi$ and $\psi_0 \in \mathcal L^2Dolb^{P,k}_\infty(\pi^*\V)$, which leads to the fact that the induced morphism in cohomology is injective.

The surjectivity can also be handled in the same way as for the lemma \ref{quasi-iso smooth subcomplex}. If $\phi \in \mathcal L ^2Dolb^{P,\bullet}(\pi^*\V)_p$ is a closed form, we fix $U$ neighbourhood of $p$ such that $\phi$ is a closed a square integrable form on $\tilde U := \pi^{-1}(U\cap M)$, by \cite[Theorem 2.12]{Bruning}, there exists a closed form $\psi$ also square integrable on $\tilde U$ in the cohomology class of $\phi$, that satisfies
\[
\psi \in \bigcap_{j\in \N} Dom(\square^j_{D''_{max}})
\]
and again since $\square_{D''_{max}}$ coincides with $\Delta_D''$ where it is defined, one obtains the surjectivity.
\end{proof}
We can now show the following theorem, which is the analog of \cite[Theorem 6.3]{Zucker79}
\begin{thm}[Dolbeault Lemma]\label{Dolbeault lemma}
The morphism between complex of sheaves of $\mathcal N (\Gamma)$-module
\[
\Omega^\bullet(\pi^*\V)_{2} \to\mathcal L^2 DR^\bullet_\infty(\pi^*\V)
\]
is a filtered quasi-isomorphism.
\end{thm}

Before doing the proof, we recall the following result from Zucker \cite[Proposition 6.4, Proposition 11.5]{Zucker79}
\begin{lem}\label{estimee d-bar}
If $(\mathcal V , h)$ is a holomorphic hermitian line bundle on a disk $\Delta_R$ of radius $R<1$, admitting a generating section $\tilde \xi$ satisfying
\[
C^{-1} r^{2\beta}|\ln r|^k \leq h(\tilde \xi,\tilde \xi) \leq C r^{2\beta}|\ln r|^k
\]
for an integer number $k$, $-1<\beta\leq 0$ with $\beta(k-1)\neq 0$ and a real constant $C>1$. Then for any square integrable holomorphic one form $fd\bar z \otimes \tilde \xi$, there exists a holomorphic one form $g\otimes e$ satisfying
\[
\begin{aligned}
(\bar \del g)\otimes \tilde \xi &= fd\bar z \otimes \tilde \xi \\
\|g\otimes \tilde \xi\|^2_2 &\leq K \|fd\bar z \otimes \tilde \xi\|^2_2
\end{aligned}
\]
where $K$ is a positive constant depending only on $\beta$, $R$ and $C$.
\end{lem}

\begin{proof}[Proof of Theorem \ref{Dolbeault lemma}]
We need to show that it induces an isomorphism 
\[
Gr^P_{\pi^*F}\Omega^\bullet(\pi^*\V)_{(2)} \to Gr^p_{\pi^*F}\mathcal L^2 DR^{\bullet}_\infty(\pi^*\V).
\]

The complex $(Gr^P_{\pi^*}\mathcal L^2 DR^\bullet_\infty(\pi^*\V),D'')$ is isomorphic to the complex of sheaves $\mathcal L^2 Dolb^{P,\bullet}_\infty(\pi^*\V)$ which is itself quasi-isomorphic to $\mathcal L^2 Dolb^{P,\bullet}(\pi^*\V)$.

The complex $\mathcal L^2 Dolb^{P,\bullet}(\pi^*\V)$ is the simple complex associated to the double complex
\[
\begin{tikzcd}
\mathcal L^2 Dolb^{0,0}(\pi^*Gr^p_{F}) \arrow[r, "\nabla' "] \arrow[d, "\bar \del"]& \mathcal L^2 Dolb^{1,0}_\infty(\pi^*Gr^{p-1}_F) \arrow[d, "\bar \del"] \\
\mathcal L^2 Dolb^{0,1}_\infty(\pi^*Gr^p_F) \ar[r, "\nabla ' "] & \mathcal L^2 Dolb^{1,1}_\infty(\pi^*Gr^{p-1}_F)
\end{tikzcd}
\]
Where $\mathcal L^2 Dolb^{r,s}(\pi^*Gr^p_{F})$  is the subsheaf of $\mathcal L^2 Dolb^\bullet(\pi^*\V)$ consisting of $\pi^*Gr^p_F$-valued $(r,s)$ forms. To have a quasi-isomorphism it suffices to check that the following sequences are exact
\[
\begin{tikzcd}[column sep = small]
0 \arrow[r] & \mathcal O(\pi^*Gr^p_F)_{(2)} \arrow[r] & \mathcal L^2 Dolb^{0,0}(\pi^*Gr^p_F) \arrow[r, "\bar \del"] & \mathcal L^2 Dolb^{0,1}(\pi^*Gr^{p}_F) \arrow[r]&0 \\
0 \arrow[r] & \Omega^1(\pi^*Gr^{p-1}_F)_{(2)} \arrow[r] &\mathcal L^2 Dolb^{1,0}(\pi^*Gr^{p-1}_F) \arrow[r, "\bar \del"]& \mathcal L^2 Dolb^{1,1}(\pi^*Gr^{p-1}_F) \arrow[r] & 0
\end{tikzcd}
\]
It follows that in the case of the point of $M$ is thus reduced to the problem of the $\bar \del$ on the disk with the Euclidean metric, in this case it is well known that the above sequences are exact.

In the case of a point $p\in X \setminus M$ the above sequence may not be exact, and we need to slightly adapt our proof. As usual we take a distinguished neighbourhood $U$ of $p$, $\tilde U$ a connected component $\pi^{-1}(U)$ and consider an $L^2$-adapted frame of $\pi^*\V_{\tilde U}$ $(\tilde \xi_1,\dots, \tilde \xi_n)$ constructed in the usual way.

Let $\eta$ be a closed square integrable one-form, and $\eta_\gamma = \eta_{|\gamma \tilde U}$ one has the decomposition by type $\eta_g = \eta_{\gamma}^{1,0} + \eta_{\gamma}^{0,1}$. One can write 
\[
\eta_{\gamma}^{0,1} = \sum_{j = 1} ^n u_jd\bar z \otimes \tilde \gamma^*\tilde \xi_j.
\]
By lemma \ref{estimee d-bar} there exists for each $j$ such that $\|\tilde \xi_j\|^2 \in W_{k}\setminus W_{k-1}$ with $k\neq 1$ there exists $v_j$ such that
\[
\begin{aligned}
\bar \del v_j\otimes \tilde \gamma^*\tilde \xi_j &= u_j \otimes \tilde \gamma^*\xi_j \\
\|v_j\otimes \gamma^*\tilde \xi_j\|^2 &\leq C \|v_j\otimes \gamma^*\tilde \xi_j\|^2
\end{aligned}
\]
We take $\nu_\gamma$ the sum of such forms, then since the basis is $L^2$ adapted one has
\[
\begin{aligned}
\|\nu_\gamma\|^2 &\leq C \|\eta_\gamma^{0,1}\|^2 \leq C \|\eta_\gamma\|^2\ \\
\|\nabla'(\nu_\gamma)\|^2 &\leq C \|\eta_\gamma^{0,1}\|^2 \leq C \|\eta_\gamma\|^2.
\end{aligned}
\]
Note that the constant $C$ does not depends on $\gamma$ since all connected component are isometric, so the form $\nu$ is square integrable and $D''(\nu)$ is also square integrable. It follows that without loss of generality by replacing $\eta$ by $\eta - D''\nu$, one can assume that the terms $\eta_\gamma^{0,1}$ only has a terms of the form $u d\bar z \otimes \tilde \xi_j$ where $\tilde \xi_j$ is a section which have asymptotic behaviour $\|\tilde \xi_j\|^2 \simeq |\ln(r)|$. The form $\eta$ is closed so
\[
\nabla'(\alpha_\gamma^{0,1}) = d'' \alpha_\gamma^{1,0}.
\]
And $\nabla'(\alpha_\gamma^{0,1})$ is a sum of terms of the form $u d\bar z \wedge \nabla (\tilde \xi_j)$, so up to a holomorphic $1$-form $\alpha_\gamma^{1,0}$ is a sum of terms of the form $u\otimes \nabla(\tilde \xi_j)$. It follows that we have up to a holomorphic one form
\[
\alpha^{1,0} + \alpha^{0,1} = D''(u\otimes \tilde \xi)
\]
up to a holomorphic one-form. This prove that the morphism
\[
Gr^p_{\pi^*F}\Omega^\bullet(\pi^*\V)_{(2)} \to \mathcal L^2 Dolb^{P,\bullet}(\pi^*\V)
\]
induces a surjection between the first cohomology groups $\mathcal H^1$

Similarly, if we have a two form $\eta$ for a fixed $\gamma$ one has
\[
\alpha_\gamma = \sum u_j d\bar z \wedge \frac{dz}{z} \otimes \gamma^*\tilde \xi_j .
\]
The lemma \ref{estimee d-bar} allow us to find a primitive $v_i dz \otimes \gamma^* \tilde \xi _j $ for each term such that $\|\tilde \tilde e _j\|^2 \simeq r^{2\beta}|\ln(r)|^k$ with $\beta(k-1) \neq 0$. The case $k = 1$ happens when $\tilde e_j\in W_{-1}(\pi^*(\V)_0)\setminus W_{-2}((\pi^*\V)_0)$ since $\|\frac{dz}{z}\|^2 = \ln(r)^2$. The operator $\nabla'$ acts as $1\wedge \frac{1}{2i\pi}\frac{dz}{z} \otimes \tilde N _0$, and 
\[
\tilde N_0 \colon W_1((\pi^*\V)_0) \to W_{-1}((\pi^*\V)_0)
\] 
is surjective, we fix $\tilde f_j$ such that $\tilde N_0 \tilde f = \tilde e _j$. And we have 
\[
2i\pi\nabla'\left(u_jd\bar z \otimes \exp\left(\frac{\tilde N_0}{2i\pi}\ln(z)\right) \gamma^*\tilde f_j\right) = u_j d\bar z\wedge \frac{dz}{z} \otimes \gamma^*\tilde \xi _j
\]
and we obtain the surjectivity in cohomology, the injectivity being clear, the morphism
\[
Gr_{\pi^*F}^P\Omega^\bullet(\pi^*\V)_{(2)} \to\mathcal L^2 Dolb^{P,\bullet}_\infty(\pi^*\V) \simeq Gr^P_{\pi^*F} L^2 DR^{P,\bullet}_\infty(\pi^*\V)
\]
 is a quasi-isomorphism and this conclude the proof of Theorem \ref{Dolbeault lemma}. This conclude the proof of Theorem \ref{Dolbeault lemma}
 \end{proof}

\section{Fredholm complexes of Hilbert $\mathcal N(\Gamma)$-modules}

In this section we recall the notion of Von Neumann dimension, for more details the reader is referred to \cite{Luck}.
We consider $\Gamma$ a discrete group, and $\ell^2(\Gamma)$ the set of complex valued functions $(a_\gamma)_{\gamma \in \Gamma}$ such that $\sum_{\gamma \in \Gamma} |a_\gamma|^2 < + \infty$. The group $\Gamma$ acts faithfully by the left regular representation $\lambda \colon \Gamma \to \mathcal{B}(\ell^2(\Gamma))$, and we denote by $\mathcal{N}(\Gamma)$ the bicommutant of $\lambda(\Gamma)$.

\begin{defn}
A Hilbert $\mathcal N (\Gamma)$-module $V$ is a Hilbert space endowed with an action of $\Gamma$ that is isometric to a a closed $\Gamma$-invariant subspace of $\ell^2(\Gamma)\hat\otimes H$ \textit{(here the notation $\hat\otimes$ designed the Hilbert tensor product)}, for some Hilbert space $H$. A morphism between two Hilbert modules is a $\Gamma$-equivariant bounded linear morphism.

We say that $V$ is finitely generated if there exists a surjective morphism $\ell^2(\Gamma)^{\oplus n} \twoheadrightarrow V$.
\end{defn}

On Hilbert $\mathcal N (\Gamma)$-module we can define a trace function $tr_{\Gamma}$ on the cone of bounded positive $\Gamma$-equivariant operator. It satisfies the usual axioms of a trace function.
\begin{itemize}
\item If $f,g$ are positive $\Gamma$-equivariant operators and $f \leq g$ then $tr_{\Gamma}(f) \leq tr_\Gamma(g)$.
\item If $f,g$ are positive $\Gamma$-equivariant operators and $\lambda \geq 0$ then $tr_{\Gamma}(f+\lambda g) = tr_{\Gamma}(f) + \lambda tr_{\Gamma}(g)$.
\item If $f$ is a positive $\Gamma$-equivariant operator then $tr_{\Gamma}(f)=0$ if and only if $f=0$.
\end{itemize}

This trace allow us to consider a Von Neumann dimension $dim_\Gamma$ of a Hilbert $\mathcal N(\Gamma)$-module by taking the trace of the identity. This notion of dimension for Hilbert $\Gamma$-module can be extended to any $\mathcal N (\Gamma)$-module (see \cite[Chapter 6]{Luck}).

\begin{prop}
The dimension function $\dim_{\Gamma}$ satisfies the following properties.
\begin{itemize}
\item If $0 \rightarrow M_0 \rightarrow M_1 \rightarrow M_2 \rightarrow 0$ is an exact sequence of $\mathcal N (\Gamma)$-module then $ \dim_{\Gamma} (M_1)   = \dim_{\Gamma} (M_2)  + \dim_{\Gamma} (M_0) $.
\item  If $0 \rightarrow M_0 \rightarrow M_1 \rightarrow M_2 \rightarrow 0$ is an weakly exact sequence of Hilbert $\mathcal N(\Gamma)$-module then $ \dim_{\Gamma} (M_1)   = \dim_{\Gamma} (M_2)  + \dim_{\Gamma} (M_0) $.
\item If $M$ is a Hilbert $\mathcal N(\Gamma)$-module then $M=0$ is and only if $\dim_{\Gamma} M = 0$.
\end{itemize}
\end{prop}

An important notion when dealing with Hilbert $\mathcal N (\Gamma)$-module is the notion of essential density.

\begin{defn}
Let $H$ be a Hilbert $\mathcal N(\Gamma)$-module, a (not necessary closed) $\Gamma$-invariant linear subspace $L\subset H$ is said to be essentially dense if for all $\varepsilon > 0$ there exists a closed sub-module $M$ of $H$ with $M\subset L$ and $dim_{\mathcal N (\Gamma)}(L^\bot) < \varepsilon$.
\end{defn}

\begin{rem}
An essentially dense subspace is dense, the converse is false. If $L\subset H$ is essentially dense, one has $\dim_{\mathcal N (\Gamma)}(H/L) = 0$.
\end{rem}

\begin{lem}
\begin{enumerate}
\item A countable intersection of essentially dense linear subspace is essentially dense (see \cite[Lemma 8.3 , (1)]{Luck}).
\item If $M\subset H$ is $\Gamma$-invariant and essentially dense in $\overline{M}$, then for all $L\subset H$  $\mathcal N(\Gamma)$ Hilbert submodule, one has $M\cap L$ is essentially dense in  $\overline M \cap L$ (see \cite[Lemma 1.17]{Shubin}).
\end{enumerate}
\end{lem}

We will work now with complexes of Hilbert $\mathcal N (\Gamma)$-modules.

\begin{defn}
A bounded complex of Hilbert $\mathcal N (\Gamma)$-module is a bounded complex $(C^\bullet,d)$, where each $C^n$ is a Hilbert $\mathcal N (\Gamma)$-module, the domain $dom(d)$ is dense and $\Gamma$-invariant and the differential $d$ is $\Gamma$-equivariant. A morphism $f\colon (C_0^\bullet,d_0) \to (C_1^\bullet,d_1)$ is given by the data of $\Gamma$-equivariant bounded morphism $f^i\colon C_{0}^i \to C_1^i$ satisfying
\[
f(dom(d_0)) \subset dom(d1), \qquad d_1f = f d_0 \quad \text{ on } dom(d_0).
\]
\end{defn}

\begin{defn}
A $\Gamma$-equivariant closed densely defined operator $f\colon H_0 \to H_1$ between Hilbert $\mathcal N(\Gamma)$-module is said to be $\Gamma$-Fredholm if there exists $\varepsilon > 0$  and $C>0$ such that for all closed $\mathcal N(\Gamma)$-submodule $L$ contained in
\[
\mathcal L(C^\bullet,\varepsilon) = \{x \in dom(f) \mid \|fx\|\leq \varepsilon \|x\| \}
\]
has dimension $\dim_{\mathcal N(\Gamma)} L < C$.

A bounded complex of Hilbert $\mathcal N(\Gamma)$-module $(C^\bullet,d)$ is said to be $\Gamma$-Fredholm if it satisfies one of the two equivalent properties below
\begin{enumerate}
\item The restriction of the differential $d$ to $im(d)^\bot$ is $\Gamma$-Fredholm.
\item If $\square_d$ is the Laplacian $\square_d := (d+d^*)^2$ and $(E_\lambda^{\square_d})_{\lambda>0}$ is its associated projection valued measure, then there exists $\varepsilon > 0$ such that $tr_{\mathcal N (\Gamma)}E_\varepsilon^{\square_d} < +\infty$.
\end{enumerate}
\end{defn}

\begin{rem}
If $C_0^\bullet$ and $C_1^\bullet$ are two complexes then $C_0^\bullet \oplus C_1^\bullet$ is $\Gamma$-Fredholm if and only if both $C_0^\bullet$ and $C_1^\bullet$ are $\Gamma$-Fredholm.
\end{rem}

For the equivalence between the two properties, by \cite[Lemma 2.3]{Luck}, a closed operator $f\colon H_1 \to H_2$ is Fredholm if and only if $E^{f^*f}_\varepsilon < +\infty$ for $\varepsilon > 0$ small enough, and we use the weak Kodaira decomposition $C^p = Ker(\square_d) \oplus \overline{Im d} \oplus \overline{Im d^*}$ so we the decomposition of the Laplacian
\[
\begin{aligned}
\square_d = d^*d_{|\overline{Im(d)}^\bot} \oplus dd^*_{|\overline{Im(d)}}
\end{aligned}
\]
gives $2. \implies 1.$. Conversely if $d$ is Fredholm then $Ker(\square_d) = Ker(d)\cap Ker d^*$ has finite dimension so $tr_{\mathcal N(\Gamma)} E_0^{\square_d} < +\infty$ and 
\[
\begin{aligned}
tr_{\mathcal N(\Gamma)} E_\varepsilon^{\square_d} - tr_{\mathcal N(\Gamma)} E_0^{\square_d} &= tr_{\mathcal N(\Gamma)} E_\varepsilon^{d^*d_{|Ker(d)^\bot}} + tr_{\mathcal N(\Gamma)}E_\varepsilon^{dd^*_{|Ker(d^*)^\bot}} \\
																		   &< +\infty
\end{aligned}
\]

\begin{lem}\cite[Lemma 2.14]{Luck}
Let $H_0,H_1$ and $H_2$ be Hilbert $\Gamma$-module, $f\colon H_0 \to H_1$ and $g \colon H_1 \to H_2$ be two closed densely-defined $\Gamma$-equivariant operators with $f(dom(f)) \subset dom(g)$ then
\begin{enumerate}
\item If $f$ and $g$ are $\Gamma$-Fredholm, $gf$ is $\Gamma$-Fredholm.
\item If $gf$ is $\Gamma$-Fredholm and $g$ is bounded then $f$ is $\Gamma$-Fredholm.
\end{enumerate}
\end{lem}

\begin{rem}
The last point implies that if $f\colon C_0^\bullet \to C_1^\bullet$ is an isomorphism of complex then $C_0^\bullet$ is $\Gamma$-Fredholm if and only if $C_1^\bullet$ is $\Gamma$-Fredholm.
\end{rem}

\begin{defn}
A quasi-isomorphism between two complexes of Hilbert $\mathcal N(\Gamma)$-modules $f\colon (C_0^\bullet,d_0) \to (C_1^\bullet,d_1)$ is a morphism of complexes of Hilbert $\mathcal N(\Gamma)$-modules inducing an isomorphism of $\mathcal N(\Gamma)$-modules on the cohomology.
\end{defn}

We can remark that a quasi-isomorphism is bounded by the definition of morphism of complexes. It is to be noted that we do not ask the induced morphism on the cohomology to be an isomorphism of topological vector spaces, but only to be an isomorphism with respect to the algebraic structure of $\mathcal N(\Gamma)$-module. We will see later that the Fredholmness property is invariant by quasi-isomorphism, to do it we begin by recalling the definition of the cone and the cylinder of a morphism of complex.

\begin{defn}
Let $(C_0^\bullet,d_0)$ and $(C_1,d_1)$ be two complexes of Hilbert of $\mathcal N (\Gamma)$-modules and $f\colon C_0^\bullet \to C_1^\bullet$ be a morphism of complexes of Hilbert $\mathcal N(\Gamma)$-modules. We define $cone(f)^\bullet$ to be the complex defined by
\[
cone(f)^n = C_0^{n+1} \oplus C_1^n \qquad d_{cone} = \begin{pmatrix}-d_0 & 0 \\ f & d_1 \end{pmatrix}
\]
where the domain of the differential of the cone is $dom(d_0)\oplus dom(d_1)$.

The cylinder of $f$ is the complex $cyl(f)^\bullet := cone\big((-id,f)\colon C_0^\bullet \to C_0^\bullet \oplus C_1^\bullet \big)$. 
\end{defn}

\begin{rem}
The usual proof shows that if $f\colon C_0^\bullet \to C_1^\bullet$ is a quasi-isomorphism of Hilbert $\mathcal N (\Gamma)$-modules then $cone(f)^\bullet$ is acyclic.
One has canonical isomorphism
\[
cyl(f)^\bullet \simeq cone\big(cone(f)^\bullet[-1] \to C_0^\bullet \big) \simeq cone\big(f \oplus 0 \colon cone(id)[-1] \to C_1^\bullet\big)
\]
\end{rem}

\begin{lem}
We have the following properties
\begin{enumerate}
\item An acyclic complex $(C^\bullet,d)$ is $\Gamma$-Fredholm and contractible i.e. there exists a morphism of complex of $\mathcal N(\Gamma)$-Hilbert module $h\colon C^\bullet \to C^\bullet[1]$ satisfying $Hd+dH = id$.
\item If $f\colon C_0^\bullet \to C_1^\bullet$ is a morphism of complexes and $C_0^\bullet$ is acyclic then there exists an isomorphism of complexes $cone(f)^\bullet \simeq C_0^\bullet[1] \oplus C_1^\bullet$.
\item If two complexes $(C_0^\bullet,d_0)$ and $(C_1^\bullet,d_1)$ are quasi-isomorphic then if one is $\Gamma$-Fredholm, the other is also $\Gamma$-Fredholm.
\end{enumerate}
\end{lem}

These results are in \cite{Luck} in the case of complexes with bounded differential.

\begin{proof}
For the first point, consider an acyclic complex $(C^\bullet,d)$, one has a closed operator $d\colon im(d)^\bot \to ker(d)$, which is both surjective and injective since the complex is acyclic. Since it has a closed image and is injective there exists $c>0$ such that for all $x\in im(d)^\bot \cap dom(d)$
\[
c\|x\| \leq \|dx\|.
\]
It follows that for all $\varepsilon < c$ one has $\mathcal L (C^\bullet,\varepsilon) = 0$ and the complex is $\Gamma$-Fredholm. To show it is contractible, set $h \colon Ker(d) \to Ker(d)^\bot \cap dom(d)$ to be equal to the inverse of the operator $d$, $h$ is bounded since it is a closed operator defined on the whole Banach space $Ker(d)$ by setting 
\[
H = h \oplus 0 \colon C^\bullet = Ker(d) \oplus Ker(d)^\bot \to C^\bullet
\]
we obtain a bounded operator such that the minimal closure of $Hd + dH$ id equal to the identity.

For the second point, one has a short exact sequence of complexes
\[
\begin{tikzcd}
0 \ar[r] & C_1^\bullet \ar[r] & cone(f)^\bullet \ar[r]&C_0^\bullet[1] \ar[r] & 0
\end{tikzcd}
\]
If $\iota \colon C_0^\bullet[1] \to cone(f)^\bullet$ is the canonical injection (which is not a morphism of complex), take $H$ the contraction of $C_O^\bullet$ defined earlier then
\[
s = d_{cone}\iota H - \iota H d_0
\]
will be a section of the previous short exact sequence and a morphism of complex provided it is bounded. So for the short exact sequence to split one only need to check that (the minimal closure of) $s$ is bounded. However it is straight forward that $Hd_0$ is equal to the orthogonal projection on $Ker(d)^\bot$, and $d_{cone}\iota H$ is equal to $-\iota\oplus fh$ which is bounded.

For the last point, fix a quasi-isomorphism $f \colon C_0^\bullet \to C_1^\bullet$, then the mapping cones $cone(f)^\bullet$ and $cone(id_{C_0^\bullet})$ are acyclic and thus $\Gamma$-Fredholm. By using the last property we have isomorphisms
\[
\begin{array}{lll}
cyl(f)^\bullet &\simeq cone\big(cone(f)^\bullet[-1] \to C_0^\bullet \big) &\simeq cone(f)^\bullet \oplus C_0^\bullet\\
		   &\simeq cone\big(f \oplus 0 \colon cone(id)[-1] \to C_1^\bullet\big) &\simeq cone(id)^\bullet[-1]\oplus C_1^\bullet.
\end{array}
\]
The isomorphisms in the first line show that $cyl(f)^\bullet$ is $\Gamma$-Fredholm if and only if $C_0^\bullet$ is $\Gamma$-Fredholm, and the isomorphisms in the second line show that $cyl(f)^\bullet$ is $\Gamma$-Fredholm if and only if $C_1^\bullet$ is $\Gamma$-Fredholm.
\end{proof}

We will use the following lemma due to Shubin.

\begin{lem}\cite[Lemma 1.15]{Shubin}
If $f\colon H_0 \to H_1$ is a closed $\Gamma$-Fredholm operator then $Im(f)$ is essentially dense in $\overline{Im(f)}$.
\end{lem}

\begin{rem}
The proof of Shubin is done in the case where $f$ is bounded. It relies on the existence the polar decomposition of $f$ which is also defined for arbitrary closed operator and the proof of Shubin can be used verbatim for the case of closed operator.
\end{rem}

It will be useful to consider the algebra of affiliated operators which allow us to forget the torsion part (see \cite[Chapter 8]{Luck}).

\begin{defn}
An affiliated operator of $\ell^2(\Gamma)$ is a closed unbounded operator $f \colon \mathcal \ell^2(\Gamma) \to \ell^2(\Gamma)$ commuting with the right regular representation of $\Gamma$, in particular this require its domain to be invariant under the right regular representation. We denote by $\mathcal U (\Gamma)$ the set of affiliated operators.
\end{defn}

\begin{lem}\cite[Lemma 8.3]{Luck}
If $L\subset \ell^2(\Gamma)$ is essentially dense then $f^{-1}(L)$ is essentially dense for $f\in \mathcal U(\Gamma)$. In particular the domain of an element of $\mathcal U (\Gamma)$ is essentially dense.
\end{lem}

It follows that if $f,g\in \mathcal U (\Gamma)$ then the minimal closure of the operator $g\circ f$ defined on the domain $g^{-1}(dom(f))$ is an affiliated operator, this gives $\mathcal U (\Gamma)$ a ring structure.
\begin{thm}\cite[Theorem 8.22]{Luck}
\begin{enumerate}
\item $\mathcal U (\Gamma)$ is a flat $\mathcal N (\Gamma)$-module 
\item $\mathcal U (\Gamma)$ is a  Von Neumann regular ring.
\item If $S\subset \mathcal N (\Gamma)$ is the set of injective operators in $\mathcal N (\Gamma)$ that have a dense image then $\mathcal U(\Gamma)$ is isomorphic to the Ore localisation $\mathcal N (\Gamma)S^{-1}$.
\end{enumerate}
\end{thm}

We have the following result, which is a particular case of \cite[Lemma 2.15]{Dingoyan}).

\begin{prop}\cite[Lemma 2.15]{Dingoyan}
If $H$ is a Hilbert $\mathcal N (\Gamma)$-module and $f\colon H \to H$ is a closed operator with an image essentially dense in its closure (for example take $f$ $\Gamma$-Fredholm), then for any $x\in \overline{Im(f)}$, there exists $r\in \mathcal N(\Gamma)$ injective with dense range such that $r\cdot x \in Im(f)$.
\end{prop}
\begin{proof}
The proof is due to Dingoyan, and holds when considering Hilbert module over a general finite Von Neumann algebra, we only write it on for the case of Hilbert $\mathcal N (\Gamma)$-module.

We begin with the case of $H = \ell^2(\Gamma)$ and $Im(f)$ essentially dense in $H$. For $y\in \ell^2(\Gamma)$, let $\rho(y)\colon \ell^2(\Gamma) \to \ell^2(\Gamma)$ be the densely-defined operator defined by $\rho(y)(z) = z * y$ (where $*$ denotes the convolution product). We fix $x\in H$ and we want to find $x_0$ such that $\rho(x_0)$ is bounded (hence belong to $\mathcal N(\Gamma)$) and $\rho(x_0)(x) \in Im(f)$.

Set $g\colon Im(f) \to Ker(f)^\bot$ to be equal to the inverse of $f$ on $Im(f)$. Since both $\rho(x)$ and $g$ commutes with the left regular representation, the composition $g\circ\rho(x)$ defines a closed densely defined operator and we can consider its polar decomposition $us$, with $u$ partial isometry and $s$ self-adjoint positive. The positivity of $s$ implies that $(1+s)$ admits a bounded inverse $(1+s)^{-1}$, and using  the identity $s(1+s)^{-1} = 1 - (1+s)^{-1}$ on a dense subset of $H$, we obtain that the minimal closure of $s(1+s)^{-1}$ is bounded thus $us(1+s)^{-1}$ is also bounded. It follows that there exists $x_0$, $x_1$ such that $(1+s)^{-1}= \rho(x_0)$, and $us(1+s)^{-1} = \rho(x_1)$ by \cite[13.8.3]{Dixmier}. By the associativity of the convolution product, one obtain on a dense subset of $H$:
\[
\rho(\rho(x_0)(x))=\rho(x)\rho(x_0) = \rho(x)(1+s)^{-1} = fg\rho(x)(1+s)^{-1} = f \rho(x_1) = \rho(f(x_1)).
\]
It follows that $\rho(x_0)(x) = f(x_1) \in Im(f)$ and $\rho(x_0)$ is bounded which proves the result for the case $H=\ell^{2}(\Gamma)$.

For the general case, one consider the injective operator
\[
\begin{array}{lll}
h \colon & Ker(f)^\bot \oplus \ell^2(\Gamma) &\to \overline{Im(f)}\oplus \ell^2(\Gamma) \\
		& (x,g) &\mapsto (f(x),g)
\end{array}
\]
Let $x \in \overline{Im(f)}$ and set $F_1 := \overline{\C[\Gamma]\cdot(x,e_\Gamma)}$, it is a Hilbert submodule of $\overline{Im(f)}\oplus \ell^2(\Gamma)$, the operator $h$ has an essentially dense image since $f$ is has an essentially dense image it follows that $im(h) \cap F_1$ is essentially dense in $F_1$, so if we set $F_0 := h^{-1}(F_1)$, the induced morphism $h\colon F_0 \to F_1$ is injective and has an essentially dense image. The morphisms $p_0\colon F_0 \to \ell^2(\Gamma)$ and $p_1\colon F_1 \to \ell^2(\Gamma)$ induced by the projection on $\ell^2(\Gamma)$ are isomorphisms of Hilbert $\mathcal N (\Gamma)$-module. So $p_1hp_0^{-1}\colon \ell^2(\Gamma) \to \ell^2(\Gamma)$ is a morphism with an essentially dense image, by the previous case there exists $r\in \mathcal N (\Gamma)$ such that $r\cdot p_1(x,e_\Gamma) = p_1(r\cdot x , r\cdot e_\Gamma) \in Im(p_1hp_0)$ so $r\cdot x \in Im(h)$ which conclude the proof.
\end{proof}

Since injective operator of $\mathcal N (\Gamma)$ that have dense image admits a left inverse in $\mathcal U (\Gamma)$, it implies the following corollary.

\begin{cor}\label{U(Gamma)-torsion}
If $(C^\bullet,d)$ is a $\Gamma$-Fredholm complex of Hilbert $\mathcal N (\Gamma)$-module, then the space $\frac{\overline{Im(d)}}{Im(d)}$ is $\mathcal U (\Gamma)$-torsion, i.e
\[
\mathcal U (\Gamma) \otimes_{\mathcal N (\Gamma)} \frac{\overline{Im(d)}}{Im(d)} = 0.
\]
\end{cor}

\section{Fredholmness of the $L^2$ De Rham complex}

We return to the setting of our paper, $X$ is a compact Riemann surface and $M\subset X$ is an open such that $D:= X \setminus M$ is a finite set of points and we consider a Galois covering $\pi\colon \tilde M \to M$ with finite monodromy at infinity and Deck group $\Gamma$. We consider a complex polarized variation of Hodge structure with underlying local system $\V$ on $M$. We begin by the following lemma.

\begin{lem}
The sheaves $\ell^2\pi^*\V, \mathcal L^2DR_\infty^\bullet(\pi^*\V), \mathcal L^2DR^\bullet(\pi^*\V), \Omega(\pi^*\V)$ and their graded admits a natural structure sheaves of $\mathcal N(\Gamma)$-modules extending their structures of sheaves of $\C[\Gamma]$-modules.
\end{lem}
\begin{proof}
The structure of $\mathcal N(\Gamma)$-module on $\ell^2\pi^*\V = j_*(\ell^2(\Gamma)\otimes_{\C[\Gamma]} \pi_!\pi^*\V)$ comes from the action of $\mathcal N(\Gamma)$ on $\ell^2(\Gamma)$.

For the other sheaves, they have a natural structure of sheaves of Fréchet space which will allows us to extend the action of $\C[\Gamma]$ to an action of $\mathcal N(\Gamma)$. We do it for the case of $\mathcal L^2DR_\infty^\bullet(\pi^*\V)$ the other being similar. Fix $U \subset X$ open, and for all $K \Subset U$ and $k\in \N^*$ we define the semi-norm
\[
\begin{array}{lrl}
\rho_{K,k} \colon &\mathcal L^2DR_\infty^\bullet(\pi^*\V)(U) &\to \R^+ \\
			  &\omega &\mapsto \|\omega\|_{2,K}^2 + \|\square_D^k\omega\|_{2,K}^2
\end{array}
\]
Those semi-norms gives $\mathcal L^2DR_\infty^\bullet(\pi^*\V)(U)$ the structure of a Fréchet space. The action of $\Gamma$ preserves the semi-norms, it will follows that the morphism $\C[\Gamma] \to \mathcal B(\mathcal L^2DR_\infty^\bullet(\pi^*\V)(U))$ is continuous for the strong topology (here $\mathcal B(F)$ denotes the set of bounded operator of a Fréchet space $F$). Since it is continuous for the strong topology it extends to a morphism $\mathcal N (\Gamma) \to \mathcal B(\mathcal L^2DR_\infty^\bullet(\pi^*\V)(U))$, and $L^2DR_\infty^\bullet(\pi^*\V)(U)$ has a structure of $\mathcal N (\Gamma)$-module, it clearly is compatible with the restriction morphism since the action of $\Gamma$ is compatible with them.
\end{proof}

All the morphisms of sheaves we considered in the previous sections where $\Gamma$-equivariant so it follows that the structure of $\mathcal N (\Gamma)$-module on the cohomology groups $H^*(X,\ell^2\pi^*\V)$ coincides with the structure induced by the $L^2$-De Rham cohomology. 

To prove this use a comparison of $L^2$ De-Rham complex to the \v{C}ech complex of $\ell^2\pi^*\V$ similar to the one done with the smooth De Rham complex in \cite[\S 8]{Bott-Tu}. This comparison will use a $L^2$-version of the \v{C}ech-De Rham complex, for thus we need a notion of double Hilbert complex.

\begin{defn}\label{double complex}
A Hilbert double complex $C^{\bullet , \bullet}$ is the data of Hilbert space $C^{p,q}$ for $(p,q) \in \Z^2$ and two closed densely defined operators $D_1 \colon C^{p,q} \to C^{p+1,q}$ and $D_2\colon C^{p,q} \to C^{p,q+1}$ such that either $D_1$ or $D_2$ is bounded and we have for all $i,j\in \{1,2\}$ $Dom(D_j) \cap Im(D_i) = D_i(dom(D_j))$, and the relations
\[
D_1^2=D_2^2=0 \qquad D_1D_2 = D_2D_1.
\]

The simple complex $(C^\bullet,D)$ associated to the double complex $(C^{\bullet, \bullet}, D_1, D_2)$ is the complex defined by 
\[
C^n = \bigoplus_{p+q = n} C^{p,q} \qquad D_{|C^{p,q}} = D_1 + (-1)^p D_2 \qquad dom(D) = dom(D_1)\cap dom (D_2).
\]
\end{defn}
 
The definition is pretty much similar to the one we have when working with abelian categories, the point worth noting in this definition is the one concerning the domains. The assumptions that one of them is bounded ensures that in the simple complex associated the differential $D$ is closed. For the other condition on the domains, the one we impose is stronger than the naive one namely $D_i(dom(D_j))\subset Dom(D_j)$. The reason is so we can have the following lemma which is usual in homological algebra (for the case of complexes in an abelian category see \cite[Lemma 8.5]{Voisin} for instance).
\begin{lem}
Let $(C^\bullet,d)$ be a bounded Hilbert complex, and $(K^{\bullet, \bullet},D_1,D_2)$ a double Hilbert complex. Assume that we have the following
\begin{enumerate}
\item There exists $i \in \Z$ such that $K^{p,q} = 0$ if $q<0$ or $p<i$.
\item We have an injective morphism of Hilbert complex $\iota \colon C^\bullet \to K^{\bullet, 0}$, which induces an isomorphism $C^n\to Ker(D_2)\cap K^{n,0}$ for each $n \in \Z$ and $Ker(D_2) \cap Dom D_1 = \iota(dom(d))$.
\item For all $p$ the morphism  $\iota\colon C^p\to(K^{p,\bullet},D_2)$ is a resolution of $C^p$.
\end{enumerate}
Then $\iota$ induces a quasi-isomorphism $\iota\colon C^\bullet \to K^\bullet$.
\end{lem}
\begin{proof}
The condition we impose on the domains will allows us to use the usual proof in case of abelian group almost verbatim, we give it nonetheless. By shifting the indices one can assume that $C^{p,q} = 0$ if $p<0$ or $q<0$. Let's begin to show that $\iota$ is injective in cohomology, take $y \in Im(\iota)\cap K^n \subset K^{n,0}$ which is $D$-exact, so there exists $x = (x_{p,q})_{p+q = n-1}$ with $Dx =y$ in particular $D_2 x_{0,n-1} =0$ so if $n-1 = 0$, $x = x_{0,0} \in Ker(D_2) = Im(\iota)$. If $n-1 > 0$, we can take $z \in K^{0,n-2} \cap dom(D_2)$ such that $D_2 z = x_{0,n-2}$, since $x_{0,n-1} \in dom(D_1)$ one can take $z \in dom(D_2)\cap dom(D_1)$ so $Dz$ is well defined (this is why we need the assumptions on the domains for the double complex) and $D(x - Dz) = Dz =y$ so one can assume $x \in \bigoplus_{q < n-1} K^{p,q}$. By induction one obtains that we can assume that $x \in K^{n-1,0}$, but in this case $D_2x = 0$ since $y\in K^{n,0}$, and $x \in Im(\iota)$ which gives us the injectivity in cohomology since $x\in dom(D_1)$ is well defined $\iota^{-1}x \in dom(d)$ since $Ker(D_2) \cap Dom D_1 = \iota(dom(d))$ which gives the injectivity.

For the surjectivity, if $y\in K^n$ is a closed form, by a similar induction one can assume $y \in K^{n,0}$, in which case $y$ is both $D_2$ and $D_1$ closed so it is the image of a $d$-closed form.
\end{proof}

Now we will construct a double Hilbert complex, whose simple complex associated will be $\Gamma$-Fredholm and quasi-isomorphic to $L^2DR^\bullet(\tilde M ,\pi^*\V)$.
For each $p\in X \setminus M$ fix an open neighbourhood $U_p$ of $p$ with $U_p \cap M$ quasi-isometric to $\Delta_{1/2}^*$ with the endowed Poincaré distance. Fix $V_p \Subset U_p$ a smaller disk and take a covering of $X \setminus \bigcup V_p$ by open disks such that their preimage by $\pi$ is a disjoint union of disks. And consider $\mathfrak U = (U_\alpha)_{\alpha\in I}$ the open cover of $X$ given by the $U_p$ and the covering of $X \setminus \bigcup V_p$. We endows $I$ with an arbitrary total order, if $\alpha_0 < \dots < \alpha_q$ are elements of $I$, we denote by $U_{\alpha_0,\dots, \alpha_q}$ the intersection $U_{\alpha_1}\cap\dots\cap U_{\alpha_q}$. We can choose the covering such that in such a way that all the $U_{\alpha_0,\dots,\alpha_q}$ are quasi-isometric to convex subsets of $\C$ for $q>0$.

Now we define the $L^2$ \v{C}ech-De Rham double complex $(K^{\bullet, \bullet},D_1,D_2)$ in the same fashion the usual one (described for instance in $\cite[\S 8]{Bott-Tu}$). So for $p,q \in \Z^2$ we set
\[
K^{p,q} =\left\{\begin{array}{ll} \underset{\alpha_0<\dots<\alpha_q \in I}{\bigoplus} L^2DR^p(\pi^{-1}U_{\alpha_0,\dots, \alpha_q}, \pi^*\V) & \text{if } (p,q) \in \N^2\\
						0 &\text{if }p < 0 \text{ or } q<0 
			\end{array}\right.
\]
the differential $D_1$ is the closed operator $D_{max}$, the maximal closure of the flat connection $D$ defined on each open subset $U_{\alpha_1, \cdots, \alpha_q}$. If $\alpha_0<\dots< \alpha_q$ are elements of $I$ fix $0\leq i\leq q$, and $\omega \in K^{p,q}$, with $\omega = (\omega_{\alpha_0,\cdots,\alpha_q})$ where the $\omega_{\alpha_0,\cdots,\alpha_q} \in L^2DR^p(\pi^{-1}U_{\alpha_0,\dots, \alpha_q}, \pi^*\V)$, we define $D_2 \omega$ by
\[
(D_2\omega)_{\alpha_0,\dots,\alpha_{q+1}} = \sum_{k=0}^{q+1} (-1)^k \omega_{\alpha_0,\dots, \alpha_{k-1},\alpha_{k+1},\dots,\alpha_{q+1}} \in L^2DR^p(\pi^{-1}U_{\alpha_0,\dots,\alpha_{q+1}}, \pi^*\V)
\]

We will verify the following lemma
\begin{lem}
The $L^2$ \v{C}ech-De Rham double complex $K^{\bullet,\bullet}$ is a Hilbert double complex in the sense of Definition \ref{double complex}, and the morphism induced by the restriction $L^2DR^\bullet(\tilde M, \V) \to K^{\bullet,0}$ induces a quasi-isomorphism between Hilbert complexes
\[
L^2DR^\bullet(\tilde M, \pi^*\V) \to K^\bullet.
\] 
\end{lem}
\begin{proof}
We have $D_2$ bounded since the covering $\mathfrak U$is finite, it is straight forward to see that one has $D_2\circ D_2 = 0 = D_1 \circ D_1$ and $D_1\circ D_2 = D_2\circ D_1$ on $dom(D_1)$.

We will construct a bounded morphism $H\colon K^{p,q+1} \to K^{p,q}$ for all $p,q \geq 0$ satisfying
\[
H(Dom(D_1)) \subset Dom(D_1) \qquad HD_2 + D_2H = id.
\]
This will prove that $Dom(D_1) \cap Im(D_2) = D_2(Dom(D_1))$, so $K^{\bullet,\bullet}$ is a double Hilbert complex, and since it is straight-forward that $L^2DR^k(M,\pi^*\V)$ is the kernel of $D_2(K^{k,0})$ the previous lemma will tell us that the restriction gives a quasi-isomorphism $L^2DR^\bullet(\tilde M, \pi^*\V) \to K^\bullet$.

This operator $H$ will be constructed in the same fashion that in the case of the smooth \v{C}ech-De Rham complex (see for instance \cite[\S8]{Bott-Tu}). We will use the same convention of $\cite{Bott-Tu}$, i.e if $\alpha_0<\dots<\alpha_q$ and $\sigma \in \mathfrak S _q$ we set
\[
\omega_{\alpha_{\sigma(0)},\dots,\alpha_{\sigma(q)}} = sign(\sigma)\omega_{\alpha_0,\dots,\alpha_q}.
\]
If we fix a partition of unity $(\rho_\alpha)_{\alpha\in I}$ subordinates to the covering $\mathfrak U$, we define $H\colon K^{p,q+1} \to K^{p,q}$ by
\[
(H\omega)_{\alpha_{0},\dots,\alpha_q} = \sum_{\alpha \in I\setminus{\alpha_0,\dots,\alpha_q}} (\rho_\alpha\circ\pi) \omega_{\alpha,\alpha_0,\dots,\alpha_q}.
\]
Where we extends the forms $\rho_\alpha \omega_{\alpha,\alpha_0,\dots,\alpha_q}$ to $U_{\alpha_{0},\dots,\alpha_q}$ by $0$. The operator $H$ is bounded and for $p,q\leq0$, one has on $K^{p,q+1}$ the relation
\[
id = HD_2 + D_2H.
\]
\end{proof}

It follows that $L^2DR^\bullet(\tilde M, \pi^*\V)$ is $\Gamma$-Fredholm if and only if the complex $K^\bullet$ is $\Gamma$-Fredholm. To check this we need to compare the complex $K^{\bullet}$ to the \v{C}ech complex $C(\mathfrak U, \ell^2\pi^*\V)$ of the sheaf $\ell^2\pi^*\V$ for the covering $\mathfrak U$.

\begin{lem}\label{Cech to De Rham}
For the covering $\mathfrak U$, the \v{C}ech complex the morphism of inclusion $C^\bullet(\mathfrak U, \ell^2\pi^*\V) \to K^{0,\bullet}$ induces a quasi-isomorphism between the complex $C^\bullet(\mathfrak U, \pi^*\V)$ and the complex $K^\bullet$. In particular $K^\bullet$ is $\Gamma$-Fredholm.
\end{lem}
Note that it does not follow from the Leray theorem since the sheaf $\ell^2\pi^*\V$ might not be acyclic in any of the neighbourhood $U_p$ for $p\in X \setminus M$.
It can also be noted that the complex $C^\bullet(\mathfrak U, \ell^2\pi^*\V)$ is $\Gamma$-Fredholm since all the $C^p(\mathfrak U, \ell^2\pi^*\V)$ are of finitely generated since $\mathfrak U$ is finite, so the last assertion is a consequence of the previous one.
\begin{proof}
It is clear that $C^q(\mathfrak U, \ell^2(\pi^*\V))$ is the kernel of $D_1\colon K^{0,q} \to K^{1,q}$.We only need to verify that $D_1$ is exact in positive degree, i.e for any open subset $U_{\alpha_0,\dots, \alpha_q}$ the Hilbert complex $L^2(\pi^{-1}U_{\alpha_0,\dots,\alpha_q},\pi^*\V)$ has no cohomology in positive degree. However any open subset $U_{\alpha_0,\dots, \alpha_q}$ is either a neighbourhood of a puncture $p \in X\setminus M$ quasi-isometric to a punctured disk $\Delta^*_R$ endowed with Poincaré metric, or $U_{\alpha_0,\dots, \alpha_q}$ is quasi-isometric to a bounded convex subset of $\C$. In the first case the complex $L^2DR^\bullet(\pi^{-1}(U_{\alpha_0,\dots,\alpha_q}))$ is acyclic as we have shown in the proof of the $L^2$-Poincaré lemma \ref{Pclem}, for the second case one has an isomorphism of Hilbert complexes
\[
L^2DR^\bullet(\pi^{-1}U,\pi^*,\V) \simeq \bigoplus L^2DR^\bullet(D, \C^n)
\]
where the $\bigoplus$ denotes the Hilbert direct sum. We know that the complex on the right has no cohomology in positive degree thanks to \cite[Lemma 4.2]{Iwaniec-Lutoborski} which gives us the wanted result. It follows that the inclusion $C^\bullet(\mathfrak U, \ell^2(\pi^*\V)) \to K^{0,\bullet}$ induces a quasi-isomorphism $C^\bullet(\mathfrak U, \ell^2(\pi^*\V)) \to K^{\bullet}$.
\end{proof}

Thanks to these lemmas, we obtain main technical result of the section.

\begin{prop}
The complex $L^2DR^\bullet(\tilde M, \pi^*\V)$ is $\Gamma$-Fredholm.
\end{prop}

This result is not a consequence of result in \cite{Eyssidieux22} where variation of Hodge structure which only implies the above Proposition when $\Sigma = \emptyset$. Moreover the equality of the Laplacians $\square_D = \square_{D''}$ implies that the Dolbeaut complex $L^2Dolb^{p,\bullet}(\pi^*\V)$ are also $\Gamma$-Fredholm. Recall that we have the Hodge to De Rham spectral sequence
\[
E_1^{p,q} = H^{p+q}(L^2Dolb_\infty^{p,\bullet}(\tilde M, \pi^*\V)) \implies H^{p+q}(L^2DR_\infty^\bullet(\tilde M,\pi^*\V)).
\]
We have the Kodaira decomposition
\[
H^{p+q}(L^2Dolb_\infty^{p,\bullet}(\tilde M, \pi^*\V)) = Harm^{p,w+q}(\tilde M, \pi^*\V) \oplus \frac{\overline{im(D'')}}{im(D'')}.
\]
The factor $Harm^{p,w+q}(\tilde M, \pi^*\V)$ is preserved at each page of the spectral sequence since harmonic forms are $D$-closed and orthogonal to $Im(D)$. Since the factor $\frac{\overline{im(D'')}}{im(D'')}$ are of $\mathcal U(\Gamma)$-torsion so we obtain finally the following result on the existence of pure Hodge structure on the cohomology.

\begin{thm}\label{Hodge structure L2 cohomology}
If $\V$ is a local system underlying a polarized complex variation of Hodge structure on $M$, the canonical isomorphisms of $\mathcal U (\Gamma)$-module
\[
\begin{aligned}
\mathcal U (\Gamma) \otimes_{\mathcal N(\Gamma)} H^*\left(X, \ell^2\pi^*\V\right) &\simeq \mathcal U(\Gamma)\otimes_{\mathcal N (\Gamma)}\mathbb{H}^*(X, \Omega_{(2)}^\bullet(\pi^*\V)) \\
 & \simeq \mathcal U(\Gamma)\otimes_{\mathcal N (\Gamma)} H_2^*(\tilde M, \pi^*\V)
\end{aligned}
\]
induces a pure Hodge structure of weight $w+k$ on $\mathcal U (\Gamma) \otimes_{\mathcal N(\Gamma)} H^k\left(X,\ell^2\pi^*\V \right)$ and one has a canonical isomorphisms
\[
\begin{aligned}
Gr^p_F(\mathcal U (\Gamma) \otimes_{\mathcal N(\Gamma)} H^k\left(X,\ell^2\pi^*\V \right)) &\simeq \mathcal U (\Gamma) \otimes_{\mathcal N(\Gamma)} \mathbb H ^k\left(X,Gr^p_F(\Omega_{(2)}^\bullet(\pi^*\V)) \right) \\
																	   &\simeq \mathcal U (\Gamma) \otimes Harm(\tilde M, \pi^*\V)^{p,w+k - p}.
\end{aligned}
\]
\end{thm}

\begin{rem}
The $\Gamma$-Fredholm property ensures that the $L^2$ cohomology groups have a finite $\mathcal N(\Gamma)$-dimension.
\end{rem}

\section{The $L^2$ Riemann-Hurwitz formula}

In this section we will show a Riemann-Hurwitz type formula. Before stating the result we set a few notations
\[
\begin{aligned}
\chi_{2,\Gamma}(\tilde M, \pi^{-1}\V) &= \sum_{k=0}^2 (-1)^k \dim_{\mathcal N(\Gamma)} H^k_{2}(\tilde M, \pi^{-1}\V)\\
\chi_{2}(M, \V) &= \sum_{k=0}^2  (-1)^k \dim H^k_{2}(M, \V)
\end{aligned}
\]
We recall that for our cohomology groups we consider we consider a hermitian metric with Poincaré singularities on $M$, and a metric given by a polarization on the flat vector bundle associated $\V$, and their pull-back and $\tilde M$. For $p\in X \setminus M$, we will denote by $n_p$ the cardinal of the isotropy subgroup of a preimage of $p$, and by $j_*\V_p$ the stalk of $j_*\V$ at $p$.

\begin{thm}($L^2$ Riemann-Hurwitz lemma)
We have the following equality.
\[
\chi_{2,\Gamma}(\tilde M, \pi^{-1}\V) - \sum_{p\in X \setminus M} \frac{\dim j_*\V_p}{n_p} = \chi_{2}(M, \V) - \sum_{p\in X \setminus M} \dim j_*\V_p
\]
\end{thm}

\begin{rem}
\begin{enumerate}
\item The integer $\chi_{2}(M, \V)$ is equal to the alternate sum of the dimension of the groups $H^k(X, j_*\V)$ so the left handside only depends on the topology of $X$ and the local system $\V$.
\item In the case $\V = \C_M$ and if $\Gamma$ is finite, we have $\chi_{2,\Gamma}(\tilde M, \pi^{-1}\V) = \frac{1}{|\Gamma|}\chi(\tilde X)$, so the formula become
\[
\chi(\tilde X) - \sum_{p\in X \setminus M} \frac{|\Gamma|}{n_p} = |\Gamma|\big( \chi(X) - |X \setminus M| \big)
\]
since $\frac{|\Gamma|}{n_p}$ is equal to the number of preimage of $p$ we recover the theorem of Riemann-Hurwitz.
\end{enumerate}
\end{rem}

\begin{proof}
By applying Lemma \ref{Cech to De Rham} to our covering and to the trivial covering one obtains.
\[
H^k(C^\bullet(\mathfrak U,\ell^2 \V)) = H^k_{2}(\tilde M, \pi^*\V) \qquad H^k(C^\bullet(\mathfrak U,j_*\V)) = H^k_{2}(M, \V)
\]
Since the \v{C}ech complex $C^\bullet(\mathfrak U,j_*\V)$ (resp. $C^\bullet(\mathfrak U,\ell^2 \V)$) is a complex of terms of finite dimension (resp. finite $\mathcal N(\Gamma)$ dimension) one obtains
\[
\begin{aligned}
\sum (-1)^k \dim_{\mathcal N(\Gamma)} H^k_{2}(\tilde M, \pi^{-1}\V) &= \sum (-1)^k \dim_{\mathcal N(\Gamma)} C^k(\mathfrak U, \ell^2\pi^*\V) \\
\sum  (-1)^k \dim H^k_{2}(M, \V) &= \sum (-1)^k \dim C^k(\mathfrak U, j_*\V) 
\end{aligned}
\]
Recall that
\[
\begin{aligned}
C^k(\mathfrak U, \ell^2\pi^*\V) &= \underset{\alpha_0<\dots<\alpha_k}{\prod} \ell^2\pi^*\V(U_{\alpha_0,\dots,\alpha_k})\\
C^k(\mathfrak U, j_*\V) &= \underset{\alpha_0<\dots<\alpha_k}{\prod} j_*\V(U_{\alpha_0,\dots,\alpha_k}).\\
\end{aligned}
\]
But for $k> 0$, one has $\ell^2\pi^*\V(U_{\alpha_0,\dots,\alpha_k}) \simeq \ell^2(\Gamma)\otimes \C^{rk(\V)}$ and $j_*\V(U_{\alpha_0,\dots,\alpha_k}) \simeq \C^{rk(\V)}$, since $U_{\alpha_0,\dots,\alpha_k}$ is a contractible and included in $M$. It follows that for $k> 0$ one has
\[
\dim_{\mathcal N(\Gamma)} C^k(\mathfrak U, \ell^2\pi^*\V) = \dim C^k(\mathfrak U, j_*\V).
\]
For $k=0$ there is two cases either $U_{\alpha_0}$ is included in $M$ and in this case quasi-isometric to a disk endowed with the Euclidean metric, or is a neighbourhood $U_p$ of a point $p$ quasi-isometric to a punctured disk $\Delta_{1/2}$ endowed with the Poincaré metric. In the first case $\ell^2\pi^*\V(U_{\alpha_0}) \simeq \ell^2(\Gamma)\otimes \C^{rk(\V)}$ and $j_*\V(U_{\alpha_0}) \simeq \C^{rk(\V)}$ ; while in the second case $j_*\V(U_p) \simeq \C^{\dim(\V_p)}$ and $\ell^2\pi^*\V(U_p) \simeq \ell^2(\Gamma/H_p)\otimes \C^{\dim(\V_p)}$ where $H_p$ is the isotropy subgroup of a preimage of $p$, note that $|H_p|=n_p$ by definition and that $\dim_{\mathcal N(\Gamma)} \ell^2(\Gamma/H_p) = \frac{1}{n_p}$. So one obtains
\[
\begin{aligned}
\chi_{2,\Gamma}(\tilde M, \pi^{-1}\V) &= \underset{U_{\alpha_0} \subset M}{\sum} rk(\V) +  \sum_{p \in X \setminus M} \frac{\dim(\V_p)}{n_p} + \sum_{k>0} (-1)^k \dim_{\mathcal N(\Gamma)} C^k(\mathfrak U, \ell^2\pi^*\V) \\
\chi_{2}(M, \V) &= \underset{U_{\alpha_0} \subset M}{\sum} rk(\V) +  \sum_{p \in X \setminus M} \dim(\V_p) + \sum_{k>0} (-1)^k \dim C^k(\mathfrak U, j_*\V)
\end{aligned}
\]
and the result follows.
\end{proof}

\section{Link with the theory of polarized Hodge module}

As before we have $j\colon M \to X$ an embedding of Riemann surfaces, with $X$ compact and $\Sigma := X \setminus M$ consisting of a finite set of points.

\subsection*{Middle extension of a variation of Hodge structure}

We set $\mathcal D _X$ the sheaf of differential operators on $X$, by a $\mathcal D _X$-module we will always mean a left $\mathcal D_X$-module. If $(\mathcal V, F^\bullet, \nabla)$ is a holomorphic vector bundle on $M$ with a flat connection, underlying a polarizable variation of Hodge structure we will denote by $\mathcal V _*$ its canonical meromorphic extension. It is endowed locally with the parabolic filtration $\mathcal V_*^\bullet$ described in section $3$. We denote by $\mathcal V_{mid}$ the $\mathcal D _X$ module generated by $\mathcal V_*^{>-1}$, i.e
\[
\mathcal V_{mid} = \sum_{j\in \N} (\nabla_{\del_z})^j \mathcal V_*^{>-1}.
\]
We call $\mathcal V _{mid}$ the middle extension of $\mathcal V$, it is endowed with a coherent good filtration $F^\bullet \mathcal V_{mid}$ defined by
\[
F^p \mathcal V_{mid} := \sum_{j\in \N} (\nabla_{\del_z})^j F^{p-j} \mathcal V^{>-1}_*.
\]

It is locally endowed with the Kashiwara-Malgrange filtration $V^\bullet \mathcal V_{mid} := \mathcal V_{mid}^\bullet$ which is given by

\[
\mathcal V _{mid}^\beta = \left\{ \begin{array}{cc} \mathcal V^\beta_* & \text{ if } \beta > -1 \\
									    (\nabla_{\del_z})^{-\lceil \beta \rceil }\mathcal V _*^{\beta - \lceil \beta \rceil} + \mathcal V_*^{> \beta} & \text{ if } \beta \leq -1
\end{array}
\right.
\]
Where $\lceil. \rceil$ denote the upper integral part of a real number, so for $\beta \in \R$ one has $\lceil\beta \rceil - 1< \beta \leq \lceil\beta \rceil$. For $\beta\in \R$, the sheaf $\mathcal V_{mid}^\beta$ is $\mathcal O_X$-coherent. The space $Gr^\beta\mathcal V_{mid}$ is finite dimensional and $(z\del_z-\beta)$ induce a nilpotent operator on $Gr^\beta \mathcal V_{mid}$, which coincides with the action of the $N_\beta$ we had on $\mathcal V_*$, and it induces a filtration $W_\bullet(N_\beta)$ on $Gr^\beta \mathcal V_{mid}$, in a similar fashion of what we did for the meromorphic extension we set
\[
M_k\mathcal V^\beta_{mid} := p^{-1}\left (W_k(N_\beta)Gr^\beta\mathcal V_{mid} \right)
\]
where $p\colon \mathcal V^\beta_{mid} \to Gr^\beta \mathcal V_{mid}$ is the projection.

\begin{prop}\cite[Proposition 6.14.2, Corollary 6.14.4]{Sabbah}
We have the following.
\begin{itemize}
\item The filtration $F^\bullet \mathcal V _{mid}$ is exhaustive.
\item For $\beta > -1$ we have $F^p\mathcal V _{mid}^\beta = (j_*F^p\mathcal V) \cap \mathcal V_*^\beta$ and $z(F^p\mathcal V_{mid}^\beta) = F^p\mathcal V_{mid}^{\beta + 1}$.
\item For $\beta \leq 0$, $\del_zF^pGr^\beta \mathcal V_{mid} = F^{p-1}Gr^{\beta-1}\mathcal V_{mid}$.
\item The sheaves $F^p \mathcal V_{mid}$, $F^p\mathcal V_{mid}^\beta$ and $F^pM_k\mathcal V_{mid}$ are $\mathcal O _X$ locally free and of finite rank.
\end{itemize}
\end{prop}

The De Rham complex of $\mathcal V_{mid}$ is given by
\[
DR(\mathcal V_{mid}) := \left\{ 0 \to \mathcal V _{mid} \overset{\nabla}{\longrightarrow} \Omega^1_X \otimes \mathcal V_{mid} \to 0 \right\}
\]
And its perverse De Rham complex is given by ${}^{p}\!DR(\mathcal V_{mid}) := DR(\mathcal V_{mid})[1]$. We have the Kashiwara-Malgrange filtration and the Hodge filtration on those bundle defined by
\[
\begin{tikzcd}
V^\beta DR(\mathcal V _{mid}) := 0 \ar[r]& \mathcal V^\beta _{mid} \ar[r, "\nabla"]& \Omega^1_X\otimes \mathcal V_{mid}^{\beta-1} \ar[r]& 0 \\
F^p DR(\mathcal V _{mid}) := 0 \ar[r]& F^p\mathcal V _{mid} \ar[r, "\nabla"]& \Omega^1_X\otimes F^{p-1}\mathcal V_{mid} \ar[r]& 0 
\end{tikzcd}
\]

One can check that by \ref{caracl2}, the holomorphic De Rham complex $\Omega^\bullet(\V)_{(2)}$ of $\mathcal V$ is a subcomplex of $V^0DR(\mathcal V_{mid})$. We recall that we have the following result \cite[Proposition 6.14.8]{Sabbah}
\begin{prop}\label{quasi-iso middle extension}
For any $p \in \Z$ the inclusions
\[
F^p\Omega^\bullet(\V)_{(2)} \to F^pV^0 DR(\mathcal V_{mid}) \to F^pDR(\mathcal V_{mid})
\]
are quasi-isomorphisms.
\end{prop}
\subsection*{Polarized Hodge module on a curve}

If $X$ is a complex manifold, a polarized Hodge module on $X$ is given by the data of $(\mathcal M, F^\bullet)$ a $\mathcal D _X$-module endowed with a good filtration, a perverse sheaf $\mathbb M^{Betti}$ endowed with a quasi-isomorphism $\alpha\colon \mathbb M^{Betti} \to {}^{p}\!DR(\mathcal M)$ and a sesquilinear pairing $S \colon \mathcal M \otimes_{\C} \overline{\mathcal M} \to \mathfrak {Db}_X$ called polarization. We require that those data satisfies some non trivial relations. In the case where $X$ is a Riemann surface polarized Hodge module admits a simple decomposition that we will recall here (see \cite[chapter 7]{Sabbah}) and that we will use as definition.

In the following $X$ will denote a Riemann surface, $\Sigma$ will be a finite set of point of $X$ and we will set $M := X \setminus \Sigma$.

\begin{defn}
A polarized Hodge module $X$ of weight $w$ with singularities at most at $\Sigma$ and of pure support $X$ is a middle extension of a polarized variation of Hodge structure on $M$ of weight $w-1$.

In this case the filtered $\mathcal D_X$ module is given by $\mathcal V _{mid}$ endowed with the Hodge filtration, and the perverse sheaf is given by $j_*\V[1]$ where $j\colon M \to X$ is the inclusion and $\V$ is the underlying local system.
\end{defn}

This definition is not the usual definition of polarized Hodge module with pure support $X$, however it is equivalent of the usual one thanks to \cite[Proposition 7.4.12]{Sabbah}. The shift of weight is there because we will consider the perverse De Rham complex ${}^{p}\!DR(\mathcal V_{mid})$ which is a shift of the usual De Rham complex.

We also need to define polarized Hodge module with punctual support $\Sigma$. Take $\mathcal H _{\Sigma}$ a sheaf of polarized Hodge structure of weight $w$ on $\Sigma$, it is equivalent to fix a polarized Hodge structure $(H_{\Sigma,p},F^\bullet H_{\Sigma,p}, h_p)$ for every point $p\in \Sigma$. We define ${}_{D}\iota(\mathcal H_{\Sigma})$ to be a skycrapper sheaf supported at $\Sigma$, the stalks at $p\in \Sigma$ is given by
\[
{}_{D}\iota \mathcal H_{\Sigma,p} := H_{\Sigma,p}[\del_z].
\]
We define $\del_z\cdot(v\del_z^k) = v\del_z^{k+1}$, for $n\in \N$ we define the action of $\C[z]$ on $H_{\Sigma,p}$ by
\[
z^n\cdot v\del_z^k = \left\{
\begin{array}{cc}
0 &\text{if }n\geq k \\
(-1)^nk(k-1)\dots(k-n+1)v\del_z^{k-n} & \text{ otherwise}
\end{array}\right.
\]
It extends naturally to an action of $\mathcal O_{X,p}$, moreover one has the relation 
\[
f\cdot(\del_z \cdot v\del_z^k) + f'\cdot v\del_z^k = \del_z\cdot(f\cdot v\del_z^k),
\]
so this endows ${}_{D}\iota(\mathcal H_{\Sigma})$ with the structure of a $\mathcal D_X$ module. It is endowed with the natural good filtration
\[
F^p{}_{D}\iota\mathcal H_{\Sigma,p} := \underset{k\in \N}{\bigoplus} (F^{p-k}H_{\Sigma,p})\del_z^k.
\]
Now we define polarized Hodge module with punctual support $\Sigma$.
\begin{defn}
A polarized (pure) Hodge module $\mathcal M$ of weight $w$ with strict support $\Sigma$ is a $\mathcal D _X$-module of the form ${}_{D}\iota\mathcal H_{\Sigma}$ where $\mathcal H_{\Sigma}$ is a sheaf of polarized Hodge structure of weight $w$ on the finite set $\Sigma$.
\end{defn}

Now we recall the support-decomposition theorem for polarized Hodge module on Riemann surfaces (see \cite[Theorem 7.4.10]{Sabbah}).
\begin{thm}
A polarized Hodge module $\mathcal M$ of weight $w$ on $X$ with singularities at most at $\Sigma$ is a $\mathcal D_X$-module $\mathcal M$ of the form $\mathcal M_1 \oplus \mathcal M_2$ where $\mathcal M_1$(resp. $\mathcal M_2$) is a polarized Hodge module of weight $w$ with pure support $X$ (resp. with punctual support $\Sigma$).
\end{thm}

\subsection*{$L^2$ direct image of polarized Hodge module}

In this section we assume that the covering $\pi\colon \tilde M \to M$ induces an unramified covering $\pi\colon \tilde X \to X$, we set $\Sigma := X\setminus M$.

We remind the following result of \cite{Eyssidieux22} that we state in the case of polarizable (pure) Hodge module. We recall the definition of the abelian category $E_f(\Gamma)$ which was define by Farber and Luck.

\begin{defn}
The category $E_f(\Gamma)$ is the category whose objects are triple $(E_1,E_2,e)$ where $E_1$ and $E_2$ are Hilbert $\mathcal N(\Gamma)$ module of finite type, and $e\colon E_1 \to E_2$ is a bounded $\mathcal N(\Gamma)$ equivariant morphism.

If $E = (E_1,E_2,e)$ and $F = (F_1,F_2,f)$ are two objects of $E_f(\Gamma)$, the set $Hom_{E_f(\Gamma)}(E,F)$ is the set of pair $(\phi_1, \phi_2)$ where $\phi_1 \colon E_1 \to F_1$, $\phi_2 \colon E_2 \to F_2$ are morphism of Hilbert $\mathcal N(\Gamma)$ module satisfying $\phi_2e = f\phi_1$, under the equivalence relation $(\phi_1,\phi_2) \sim (\psi_1,\psi_2)$ if there exists $T\colon F_2 \to E_1$ such that $\psi_2 - \phi_2 = fT$. 
\end{defn}

\begin{prop} \cite{Farber}
The category $E_f(\Gamma)$ is abelian.
\end{prop}

\begin{rem}
The category of Hilbert $\mathcal N (\Gamma)$ module is embedded in the category of Farber via the functor
\[
E \to (\{0\},E,0).
\]
Moreover one has a forgetful functor from the category $E_f(\Gamma)$ to the category of $\mathcal N(\Gamma)$-module of finite type given by
\[
(E_1,E_2,e) \to E_2/e(E_1)
\]
\end{rem}

With this we can state the following result of \cite{Eyssidieux22}.

\begin{thm}\cite[Corollary 2]{Eyssidieux22}
Let $(X,\omega_X)$ be a compact Kähler manifold, $\pi\colon \tilde X \to X$ a Galois group of covering group $\Gamma$. We denote by $pHM(X)$ the category of polarized Hodge module on $X$. There exists a $\del$-functor satisfying Atiyah index theorem and Poincaré-Verdier Duality
\[
L^2dR\colon D^bpHM(X) \to D^bE_f(\Gamma)
\]
such that for all Hodge module $\mathbb M$ and $q\in \Z$
\[
H^q(L^2dR(\mathcal M)) \simeq H^q_{(2)}(\tilde X,\pi^*\mathbb M ^{Betti}).
\]
Those cohomology groups are endowed with a filtration $F^\bullet$ induced by Saito's Hodge filtration on $\mathcal M$.
\end{thm}

We recall the definition the functor $L^2dR$ and some of its properties. For this we need to recall the definition of the functor $\ell^2\pi_{*}$ of $L^2$-direct image as it is introduced in \cite{Campana-Demailly, Eyssidieux00}. It is a functor from the category of coherent $\mathcal O_X$ module to the the category of $\mathcal N(\Gamma)_X$-module where if $\mathcal F$ is a coherent $\mathcal O_X$-module, $U$ is coordinates chart such that $\pi^{-1}(U) \simeq \Gamma \times U$ and we fix $\phi\colon \mathcal O_{X\mid U}^{\oplus N} \to \mathcal F_{U}$ a local presentation
 \[
\ell^2\pi_{*} \mathcal F (U) := \left\{ (s_\gamma)_\gamma \in \mathcal F(U)^\Gamma \left | \begin{array}{ll} \exists (r_\gamma)_\gamma \in \mathcal O_X(U)^\Gamma, & \forall K \subset U \text{ compact}\\ \sum_{\gamma\in \Gamma} \int_K |r_\gamma|^2 <+\infty& \end{array} \right. \right\}.
\]
It can be check that it does not depends of the local presentation and since $X$ is compact, $\ell^2\pi_{*}$ does not depends on the Kähler metric $\omega_X$, moreover by definition $\ell^2\pi_{*}\mathcal F$ is a subsheaf of $\pi_*\pi^* \mathcal F$.
\begin{lem}\cite[Lemma 2.1.1]{Eyssidieux22}
The functor $\ell^2\pi_{*}$ can be extended in a exact functor on the category $Mod(\mathcal O_X)$ of $\mathcal O_X$-module by setting
\[
\ell^2\pi_{*}\mathcal F := \ell^2\pi_*\mathcal O_X \otimes_{\mathcal O _X}\mathcal F.
\]
\end{lem}

We need to extend this functor in the category of $Qcoh(\mathcal O_X, Diff_X)$ of quasi-coherent $\mathcal O_X$-module whose morphisms are differential operators. We recall that a differential operator $P\colon \mathcal F _1 \to \mathcal F_2$ is an operator in the image of
\[
\nu \colon Hom_{\mathcal O_x}(\mathcal F _1, \mathcal F_2 \otimes_{\mathcal O_X} \mathcal D_X) \to Hom_{\C _X}(\mathcal F_1, \mathcal F_2)
\]
where the structure of $\mathcal O_X$-module on $\mathcal F_2 \otimes_{\mathcal O_X} \mathcal D_X$ is given by the right $\mathcal O_X$-module structure, and $\nu$ is given by the left composition by the morphism
\[
\begin{aligned}
\mathcal F_2 \otimes_{\mathcal O_X} \mathcal D_X &\to \mathcal F _2 \\
f\otimes P &\mapsto P(1)f
\end{aligned}
\]
\begin{lem}\label{functoriality l2}\cite[Lemma 2.2.2]{Eyssidieux22}
Let $P := \nu(p) \colon \mathcal F_1 \to \mathcal F_2$ be a differential operator, the morphism $\ell^2\pi_*P := \nu(\ell^2\pi_*p)\colon \ell^2\pi_{*}\mathcal F_1$ to $\ell^2\pi_{*}\mathcal F _2$ is the restriction of $\pi_*\pi^*P\colon \pi^*\pi_* \mathcal F_1 \to \pi_*\pi^*\mathcal F_2$ and it is a morphism of sheaves of $\mathcal N(\Gamma)$-module.  Hence $\ell^2\pi_{*}$ defines an additive functor from the category $QCoh(\mathcal O_X, Diff_X)$ to the category $Mod(\mathcal N(\Gamma)_X)$.
\end{lem}

With this lemma one can define $L^2dR$ to be the functor $R\Gamma\circ \ell^2\pi_*{}^{p}\!DR(.)$. On the Betti side one can understand it as the functor that sends $\mathbb M ^{Betti}$ to $\ell^2\pi_*\mathbb M^{Betti} := \ell^2(\Gamma)\otimes_{\C[\Gamma]}\pi_!\pi^*\mathbb M ^{Betti}$.

\begin{prop}
Let $\mathcal V_{mid}$ be a polarized Hodge module on $X$ with pure support $X$ and weight $w$. One has a natural filtered inclusions
\[
F^\bullet \Omega^\bullet(\pi^*\V)_{(2)} \to \ell^2\pi_*F^\bullet DR(\mathcal V_{mid})
\]
which is filtered quasi-isomorphism.
\end{prop}
\begin{proof}
We recall that in the case of the trivial covering the following inclusion is a quasi-isomorphism for all $p\in \Z$ by Proposition \ref{quasi-iso middle extension}
\[
F^p\Omega^\bullet(\V)_{(2)} \to F^pDR(\mathcal V_{mid}).
\]
We want to check that we obtain quasi-isomorphisms
\[
F^p\Omega^\bullet(\pi^*\V)_{(2)} \to \ell^2\pi_*F^pDR(\mathcal V_{mid}).
\]
The result will follow from the functoriality of $\ell^2\pi_*$ given by the previous lemma \ref{functoriality l2}. One has a natural isomorphism
\[
\Omega^\bullet(\pi^*\V)_{(2)} \to \ell^2\pi^* \Omega(\V)_{(2)}
\]
so one has to check that $\ell^2\pi_*\nabla$ induces a isomorphism
\[
[\ell^2\pi_*\nabla] \colon \ell^2\pi_*\left(\bigslant{F^pV^0\mathcal V _{mid}}{F^p\mathcal O(\V)_{(2)}}\right) \to \ell^2\pi_*\left(\bigslant{F^{p-1}V^{-1}\mathcal V _{mid}}{F^p\Omega^1(\V)_{(2)}}\right)
\]
which is true as $[\ell^2\pi_*\nabla]$ coincides with $\ell^2\pi_*[\nabla]$ where 
\[
[\nabla] \colon \bigslant{F^pV^0\mathcal V _{mid}}{F^p\mathcal O(\V)_{(2)}} \to \bigslant{F^{p-1}V^{-1}\mathcal V _{mid}}{F^p\Omega^1(\V)_{(2)}}
\]
is the morphism induced by $\nabla$, and $[\nabla]$ is an isomorphism by Proposition \ref{quasi-iso middle extension}.
\end{proof}

Finally one obtains the following result which gives a positive answer to \cite[Conjecture 3]{Eyssidieux22} for the case of pure Hodge module on Riemann surfaces.
\begin{thm}
If $\mathbb M$ is a Hodge module on a compact Riemann surface $X$, then the groups  $\mathcal U(\Gamma)\otimes_{\mathcal N(\Gamma)}\mathbb H^k(L^2dR(\mathbb M))$ is admits a pure Hodge structure of weight $w+k$ in the abelian category of $\mathcal U(\Gamma)$-modules. The Hodge filtration is induced by Saito's Hodge filtration on the perverse De Rham complex ${}^{p}\!DR(\mathcal M)$. 

In the case where $\mathbb M$ is a polarized Hodge module with pure support $X$ and with singularities at most at $\Sigma := X\setminus M$ there is an isomorphism of $\mathcal U(\Gamma)$-module
\[
\mathcal U(\Gamma)\otimes_{\mathcal N(\Gamma)}\mathbb H^k(L^2dR(\mathbb M)) \simeq \mathcal U(\Gamma)\otimes_{\mathcal N(\Gamma)}H^k_{(2),red}(\tilde M,\pi^*\V,\pi^*\omega_{Pc},\pi^*h)
\]
which is an isomorphism of Hodge structure, where the Hodge structure on the reduced $L^2$-cohomology comes from the decomposition of harmonic forms by type.
\end{thm}
\begin{proof}
The above proposition and the Theorem \ref{Hodge structure L2 cohomology} treat the case where $\mathbb M$ has pure support $X$, so in the following we can assume that $\mathbb M$ has punctual support $\Sigma$. On this case $\mathbb M$ is of the form $\iota \mathcal H_\Sigma$ where $\mathcal H_\Sigma$ is a sheaf of polarized Hodge structure of weight $w$ on $\Sigma$. In this case $H^k(L^2dR(\mathbb M)) = 0$ unless $k=0$ in which case it is equal to
\[
\bigoplus_{p\in \Sigma} \ell^2(\Gamma)\otimes_{\C}H_{\Sigma,p}.
\]
so the cohomology is fully reduced and one has a Hodge structure on $H^0(L^2dR(\mathbb M))$ of weight $w$ on the category of $\mathcal N (\Gamma)$-module given by the Hodge structure on each $H_{\Sigma,p}$, this gives the wanted Hodge structure on $\mathcal U (\Gamma)\otimes H^0(L^2dR(\mathbb M))$.
\end{proof}
\bibliographystyle{alpha}
\bibliography{Reference.bib}

\textbf{Bastien, Jean}

\textbf{Institut Fourier}, UMR 5582, Laboratoire de Mathématiques

\textbf{Université Grenoble Alpes}, CS 40700, 38058 Grenoble cedex 9, France

\textit{E-mail:} bastien.jean1@univ-grenoble-alpes.fr
\end{document}